\documentclass{article}
\usepackage{bm}
\usepackage{tikz,mathpazo}
\usepackage{pgf}
\usepackage[top=2.5cm, bottom=2.5cm, left=3cm, right=3cm]{geometry}   
\usepackage{indentfirst}
\usepackage{graphicx}
\usepackage{graphics}
\usepackage[toc,page,title,titletoc,header]{appendix}
\usepackage{bm}
\usepackage{listings}  
\usepackage{amsmath}
\usepackage{setspace} 
\usepackage{indentfirst}
\usepackage{caption}
\usepackage{multirow} 
\usepackage{lipsum,multicol}
\usepackage{pdfpages}
\usepackage{float}
\usepackage{amsthm}
\usepackage{pdfpages}
\usepackage{url}
\usepackage{colortbl}
\usepackage{subfigure}
\usepackage{epsfig}
\usepackage{epstopdf}
\usetikzlibrary{shapes.geometric, arrows}
\usepackage{fancyhdr}
\usepackage{abstract}
\usepackage{tikz,mathpazo}
\usetikzlibrary{shapes.geometric, arrows}

\usepackage{amssymb}
\usepackage{latexsym}
\usepackage{verbatim}
\usepackage[numbers]{natbib} 
\usepackage{booktabs}

\usepackage{tikz}

\allowdisplaybreaks[4]

\newcommand{\inner}[1]{\left\langle #1 \right\rangle}
\newcommand{\norm}[1]{\left\Vert #1\right\Vert}
\newcommand{\bb}[1]{\mathbb{#1}}

\newcommand{\ca}[1]{\mathcal{#1}}
\newcommand{\tr}[0]{\mathrm{tr}}

\newcommand{\tX}[0]{\tilde{X}}

\newcommand{\Diag}[0]{\mathrm{Diag}}

\newcommand{\Xo}[0]{{X_{ \mathrm{orth} }} }

\newcommand{\ff}{_{\mathrm{F}}}
\newcommand{\fs}{^2_{\mathrm{F}}}
\newcommand{\tp}{^\top}

\newcommand{\A}{\ca{A}}

\newcommand{\prox}{{\mathrm{prox}}}
\newcommand{\Apen}[1]{\left( \frac{3}{2} I_p - \frac{1}{2} {#1}\tp {#1} \right)}

\newcommand{\SLPGs}{{SLPG }}

\newcommand{\Xk}{{X_{k} }}
\newcommand{\Yk}{{Y_{k} }}

\newcommand{\Xkp}{{X_{k+1} }}

\newcommand{\Dk}{{D_{k} }}
\newcommand{\etak}{{\eta_{k} }}
\newcommand{\Ykp}{{Y_{k+1} }}

\newcommand{\grad}{{\mathit{grad}\,}}

\newcommand{\ProjS}{{\ca{P}_{\ca{S}_{n,p}}}}

\newcommand{\partialRie}{ {\partial}_{\ca{R}} }
\newcommand{\partialProj}{{\partial}_{\ca{P}}}
\newcommand{\JA}{\ca{J}_{\mathcal{A}}}
\newcommand{\D}{\ca{D}}

\newcommand{\hsharp}{{h^{\sharp}}}

\newtheorem{theo}{Theorem}[section]
\newtheorem{lem}[theo]{Lemma}
\newtheorem{prop}[theo]{Proposition}

\newtheorem{cond}[theo]{Condition}
\newtheorem{coro}[theo]{Corollary}

\newtheorem{defin}[theo]{Definition}
\newtheorem{rmk}[theo]{Remark}
\newtheorem{assumpt}[theo]{Assumption}

\usepackage{caption}
\usepackage{algorithm}
\usepackage{algpseudocode}
\usepackage{longtable}
\usepackage{appendix}

\numberwithin{equation}{section}

\title{A Constraint Dissolving Approach for Nonsmooth Optimization over the Stiefel Manifold}

\author{ Xiaoyin Hu
	\thanks{School of Computer and Computing Science, Hangzhou City University, Hangzhou, 310015, China (huxy@zucc.edu.cn). This research was supported by Zhejiang Provincial Natural Science Foundation of China under Grant (No. LQ23A010002) and Scientific research project of Zhejiang Provincial Education Department(Y202248716).},~
	Nachuan Xiao
	\thanks{The Institute of Operations Research and Analytics, National University of Singapore, Singapore (xnc@lsec.cc.ac.cn). The research of this author is supported by the Ministry of Education, Singapore, under its Academic Research Fund Tier 3 grant call (MOE-2019-T3-1-010).},~
	Xin Liu
	\thanks{State Key Laboratory of Scientific and Engineering Computing, Academy of Mathematics and Systems Science, Chinese Academy of Sciences, and University of Chinese Academy of Sciences, China (liuxin@lsec.cc.ac.cn). Research is supported in part by the National Natural Science Foundation of China (No. 12125108, 11971466, 12288201, 12021001 and 11991021), Key Research Program of Frontier Sciences, Chinese Academy of Sciences (No. ZDBS-LY-7022).},~
	and Kim-Chuan Toh
	\thanks{Department of Mathematics, and Institute of Operations Research and Analytics, National University of Singapore, Singapore 119076 (mattohkc@nus.edu.sg). The research of this author is supported by the Ministry of Education, Singapore, under its Academic Research Fund Tier 3 grant call (MOE-2019-T3-1-010).}
}

\begin{document}
	\maketitle
	
	\begin{abstract}
		This paper focus on the minimization of a possibly nonsmooth objective function over the Stiefel manifold. The existing  approaches either lack efficiency or can only tackle prox-friendly objective functions.  We propose a constraint dissolving function named \ref{NEPen} and show that it has the same first-order stationary points and local minimizers as the original problem in a neighborhood of the Stiefel manifold. Furthermore, we show that the Clarke subdifferential of \ref{NEPen} is easy to achieve from the Clarke subdifferential of the objective function. Therefore, various existing approaches for unconstrained nonsmooth optimization can be directly applied to nonsmooth optimization problems over the Stiefel manifold. We propose a framework for developing subgradient-based methods and establish their convergence properties based on prior works. Furthermore, based on our proposed framework, we can develop efficient approaches for optimization over the Stiefel manifold. Preliminary numerical experiments further highlight that the proposed constraint dissolving approach yields efficient and direct implementations of various unconstrained approaches to nonsmooth optimization problems over the Stiefel manifold. 
	\end{abstract}

	\section{Introduction}
	In this paper, we consider the following nonsmooth optimization problem
	\begin{equation}
		\label{Prob_Ori}
		\tag{OCP}
		\begin{aligned}
			\min_{X \in \bb{R}^{n\times p}} \quad &f(X)\\
			\text{s. t.} \quad & X\tp X = I_p,
		\end{aligned}
	\end{equation}
	where the objective function $f: \bb{R}^{n\times p} \mapsto \bb{R}$ is locally Lipschitz continuous and possibly nonsmooth, while  $I_p$ denotes the $p\times p$ identity matrix. 
	We denote the feasible region of \ref{Prob_Ori} as $ \ca{S}_{n,p} := \{X \in \bb{R}^{n\times p}:  X\tp  X = I_p \}$, which is also referred as the Stiefel manifold throughout this paper. 
	
	Our interest in \ref{Prob_Ori} with nonsmooth objective functions comes from its various applications in different engineering fields, including sparse principal component analysis \cite{chen2018proximal,xiao2020l21}, robust subspace recovery \cite{lerman2018overview,wang2019globally},  packing problems \cite{absil2017collection}, and  training orthogonally constrained  neural networks  \cite{arjovsky2016unitary,bansal2018can}, where the objective function is only locally Lipschitz continuous and usually not Clarke regular (the definition for Clarke regular can be found at Definition \ref{Defin_regular}). As an illustration, let us present some motivating applications arisen from training deep neural networks.
	
	Training deep neural networks often become challenging both theoretically and practically \cite{glorot2010understanding}
	when the phenomenon of gradients vanishing or exploding happens.
	To address this issue, several recent works focus on imposing orthogonality constraints to the weights of the layers in these deep neural networks \cite{arjovsky2016unitary,bansal2018can,lezcano2019trivializations}. As orthogonality implies energy preservation properties  \cite{zhou2006special}, these existing works demonstrate that the orthogonal constraints can stabilize the distribution of activations over layers within convolutional neural networks and make their optimization more efficient. Moreover, some existing works \cite{lezcano2019cheap,lezcano2019trivializations,wang2020orthogonal} observe encouraging improvements in the accuracy and robustness of the networks with orthogonal constraints.  
	However, when these neural networks are built from nonsmooth activation functions, their loss functions are usually not Clarke regular.   For example, the loss function for fully connected two-layer neural network with mean squared error (MSE) can be expressed as 
	\begin{equation*}
		f_{\text{mse}}(W_1, W_2, b_1, b_2)\,=\,\frac{1}{N}\sum_{i = 1}^{N} ~ \norm{ \tilde{\sigma}_2(W_2\tilde{\sigma}_1(W_1x_i + b_1) + b_2) - y_i }_2^2,
	\end{equation*}
	where $\{W_1, W_2\}$ denote the weight matrices, $\{b_1, b_2\}$ denote the bias vectors, $\{\tilde{\sigma}_1, \tilde{\sigma}_2\}$ refer to the activation functions, and $\{(x_i, y_i)\mid i = 1,...,N\}$ is the training data set.  
	When $\tilde{\sigma}_i$ ($i=1,2$) are chosen as  nonsmooth activation functions, such as  rectified linear unit (ReLU) and leaky ReLU \cite{maas2013rectifier}, the loss function is clearly locally Lipschitz continuous but not Clarke regular
	with respect to $(W_1, W_2, b_1, b_2)$. 
	Therefore, we only assume that the objective function $f$ in \ref{Prob_Ori} to be locally Lipschitz continuous throughout this paper.

	\subsection{Existing works}
	Minimizing smooth objective functions over the Stiefel manifold has been extensively studied in the past several years. 
	Interested readers can refer to the books \cite{Absil2009optimization,boumal2020introduction}, the recent survey paper \cite{hu2020brief} and the references therein for more details. 
	Compared to smooth optimization over the Stiefel manifold, research on nonsmooth cases are limited \cite{li2021weakly}. 
	In the following, we briefly mention some existing state-of-the-art approaches for nonsmooth optimization over the Stiefel manifold.

	\paragraph{Riemannian subgradient methods}
	
	Some existing Riemannian subgradient methods are specifically designed for minimizing geodsically convex objectives over a Riemannian manifold \cite{ferreira1998subgradient,ferreira2002proximal,zhang2016first}. However, the geodesic convexity over a compact Riemannian manifold, such as the Stiefel manifold discussed in this paper, only holds for constant functions. Therefore, these methods are invalid for any nontrivial case of \ref{Prob_Ori}.
	Very recently, \cite{li2021weakly} studies a class of Riemannian subgradient methods to minimize weakly convex functions over the Stiefel manifold. By extending the subgradient inequality \cite{davis2019stochastic} from $\bb{R}^{n\times p}$ to the tangent spaces of the Stiefel manifold, the theoretical analysis of these Riemannian subgradient methods can be implemented by the same methodologies as existing works for unconstrained optimization \cite{davis2018subgradient,davis2019stochastic,li2019incremental}. However, the analysis in \cite{li2021weakly} relies on the weak convexity of the objective function, which we do not assume for \ref{Prob_Ori}.  Therefore, when we apply subgradient methods to solve \ref{Prob_Ori}, the proposed frameworks in \cite{li2021weakly} are not capable of analyzing the related theoretical properties.

	\paragraph{Riemannian gradient sampling methods}
	
	Gradient sampling methods \cite{burke2005robust,burke2020gradient}
	are originally proposed for unconstrained nonsmooth nonconvex optimization. For nonsmooth optimization problems over a Riemannian manifold, several recent works \cite{hosseini2018line,hosseini2017riemannian} proposed Riemannian gradient sampling methods based on their unconstrained origins. When applied to the minimization over the Stiefel manifold, those Riemannian gradient sampling methods first take a few sampling points $\{ X_{k,j}\mid j = 1,...,N_J\} \subset \ca{S}_{n,p}$ in a neighborhood of the current iterate $X_k$. Then they aim to find a descent direction in the convex hull of the Riemannian gradient at  the sample points
	$\{ X_{k,j} \mid j = 1,...,N_J\}$. However, as mentioned in \cite{hosseini2017riemannian,hosseini2018line}, the number of sample points $N_J$ should be much larger than the dimension of the manifold. Therefore, for optimization over the Stiefel manifold, the  number of sampling points should be much larger than $np-p(p+1)/2$ as described in \cite{li2021weakly,burke2020gradient}. As emphasized in \cite{li2021weakly}, it is usually expensive to generate a descent direction in existing Riemannian gradient sampling methods, especially in high dimensional cases.

	\paragraph{Proximal gradient methods}	
	
	For problems with prox-friendly objective functions, for instance, the summation of a
	smooth function and a nonsmooth regularizer whose proximal mapping is easy-to compute,
	the Riemannian proximal gradient methods are developed by extending the proximal gradient methods from Euclidean space to Riemannian manifolds. Several existing works \cite{ferreira2002proximal,de2016new} develop the Riemannian proximal gradient methods by computing the proximal mappings over the Riemannian manifold. Although their global convergence properties could be established by following the same techniques as their unconstrained counterparts, computing the proximal mappings in these approaches is expensive, except for some special cases. Hence these approaches are generally inefficient in practice as discussed in \cite{li2021weakly}. Recently, several Riemannian proximal gradient methods \cite{chen2018proximal,huang2019riemannian,zhou2021semi} are proposed by computing the proximal mapping in the tangent space rather than over the manifold. Therefore, computing the proximal mappings in each iteration is equivalent to solving a linearly constrained strongly convex optimization problem, which usually does not have closed-form solution. Although various existing iterative methods such as semi-smooth Newton methods \cite{qi1993nonsmooth} and Arrow-Hurwicz methods \cite{chambolle2011a} can be applied for the proximal subproblem, solving the proximal subproblem is usually a bottleneck in these Riemannian proximal gradient methods.

	Recently, a novel class of efficient approaches for Stiefel manifold optimization are developed based on  exact penalty models. Inspired by the exact penalty model (PenC) for smooth optimization over the Stiefel manifold \cite{xiao2020class,hu2020anefficiency}, \cite{xiao2020l21} extends PenC to $\ell_{2,1}$-norm regularized cases and proposes a proximal gradient method called PenCPG.
	In PenCPG, the proximal subproblem has a closed-form solution, which leads 
	to its numerical superiority over existing Riemannian proximal gradient approaches. 
	
	Furthermore, for generalized composite objective functions, \cite{xiao2021penalty} 
	proposes a penalty-free infeasible approach named sequential linearized proximal gradient method (SLPG). 
	In each iteration, SLPG alternatively takes the tangential and the normal steps, both of which do not involve any
	penalty parameter or orthonormalization procedure. Consequently, \SLPGs enjoys high scalability and avoids the numerical inefficiency from inappropriately selected penalty parameters. 
	
	However, the efficiencies of all the aforementioned proximal gradient methods heavily rely on having
	prox-friendly objective functions. 
	For general nonsmooth objective functions, computing its corresponding proximal mappings can be extremely expensive \cite{li2021weakly}.  As a result, we need new approaches for general nonsmooth optimization over the Stiefel manifold.

	\subsection{Motivation}

	Our motivation comes from the penalty function approaches for Riemannian optimization. For smooth optimization problems over the Stiefel manifold, \cite{xiao2021solving} proposes an exact penalty  function (ExPen) that takes the form as follows:
	\begin{equation}
		\tag{ExPen}
		\label{ExPen}
		f\left( X\Apen{X}\right) + \frac{\beta}{4} \norm{X\tp X - I_p}\fs.
	\end{equation}
	They show that \ref{ExPen} is an exact penalty function for smooth optimization over the Stiefel manifold. Different from the existing Fletcher's penalty function \cite{fletcher1970class}, the objective function in \ref{ExPen} does not involve $\nabla f$. Therefore, \ref{ExPen} has easy-to-compute derivatives and enables direct implementation of various existing unconstrained approaches for solving smooth optimization problems over the Stiefel manifold. 
	
	Very recently,  \cite{xiao2022constraint} proposes constraint dissolving approaches for minimizing smooth functions over a closed Riemannian manifold. In their proposed approaches, solving a Riemannian optimization problem is transferred into  the unconstrained minimization of a corresponding constraint dissolving function. According to \cite{xiao2022constraint}, the  constraint dissolving function for \ref{Prob_Ori} can be expressed as 
	\begin{equation}
		\tag{CDF}
		\label{CDF}
		f\left( \A(X) \right) + \frac{\beta}{4} \norm{X\tp X - I_p}\fs,
	\end{equation}
	where the constraint dissolving mapping $\A: \bb{R}^{n\times p} \to \bb{R}^{n\times p}$ is locally Lipschitz smooth and  satisfies the following conditions.
	
	\begin{cond}\label{cond:1}
		\begin{enumerate}
			\item $\A(X) = X$ holds for any $X \in \ca{S}_{n,p}$.
			\item The Jacobian of $\A(X)\tp \A(X) - I_p$ equals to $0$ for any $X \in \ca{S}_{n,p}$. 
		\end{enumerate}
	\end{cond}
	
	Condition \ref{cond:1}
	grants great flexibility in designing the constraint dissolving functions for \ref{Prob_Ori}. 
	However, for nonsmooth optimization problems over the Stiefel manifold, the chain rule usually fails when computing $\partial (f\circ \A)$. Therefore, the Clarke subdifferential of \ref{CDF} is usually not easy to obtain  from $\partial f$ and the Jacobian of $\A$. 
	In particular, the mapping $X \mapsto X(\frac{3}{2} I_p - \frac{1}{2}X\tp X)$, adopted in ExPen, 
	is not a homeomorphism over $\bb{R}^{n\times p}$. Therefore, computing the Clarke subdifferential for \ref{ExPen} is usually not easy when $f$ is not Clarke regular. That is the reason why we can not directly extend \ref{ExPen} to solve \ref{Prob_Ori}.

	Nevertheless, we notice that the chain rule can be guaranteed 
	if $\A$ is a {\it homeomorphism} (i.e. $\A^{-1}$ is well-defined and continuous over $\bb{R}^{n\times p}$) \cite[Theorem 2.3.10]{clarke1990optimization},
	even in the case that $f$ is not Clarke regular. The remaining question now is  how to construct a homeomorphic $\A$.

	\subsection{Contribution}
	Taking both Condition \ref{cond:1} and the desirable homeomorphism property of $\A$ into account, 
	we construct the following nonsmooth constraint dissolving function named \ref{NEPen}
	for nonsmooth optimization over the Stiefel manifold, 
	\begin{equation}
		\label{NEPen}
		\tag{NCDF}
		h(X):= f\left( \A(X) \right) + \frac{\beta}{4} \norm{X\tp X - I_p}\fs,
	\end{equation}
	where $\A: \bb{R}^{n\times p} \to \bb{R}^{n\times p}$ is defined as
	\begin{equation}
		\A(X) := \frac{1}{8}X \left( {15} I_p - {10}X\tp X + 3 (X\tp X)^2  \right).
		\label{A-map}
	\end{equation} 
	We prove that $\A$ is a homeomorphism and $\partial (f\circ \A)$ follows the chain rule \cite[Theorem 2.3.10]{clarke1990optimization},  and  $\partial h(X)$ has an explicit expression 
	which is easy to achieve from the subdifferential of $f$. Based on the explicit expression of $\partial h(X)$, we prove that  \ref{Prob_Ori} and \ref{NEPen} share the same local minimizers and stationary points under mild assumptions for any $\beta > 0$. These properties characterize the equivalence between \ref{Prob_Ori} and \ref{NEPen} and further illustrate that various unconstrained approaches can be directly implemented to solve \ref{Prob_Ori} through \ref{NEPen}.

	Moreover, we present several examples to show that \ref{Prob_Ori} can be solved by directly applying subgradient methods to minimize \ref{NEPen}. In these examples, we propose a framework for developing subgradient methods for \ref{Prob_Ori}, and show that their convergence properties can be established directly from existing rich theoretical results in unconstrained optimization. Furthermore, we propose a novel proximal subgradient method for \ref{Prob_Ori}, and prove its convergence properties  based on \ref{NEPen} and our proposed framework. 
	Preliminary numerical experiments on training neural networks with orthogonally constrained weights further illustrate the superior performance of our proposed algorithms, especially in the aspect of computational efficiency.   These illustrative examples highlight the promising potential of \ref{NEPen}.

	\subsection{Notations}
	Throughout this paper, the Euclidean inner product of two matrices $X, Y\in \bb{R}^{n\times p}$ is defined as $ \inner{X, Y}=\tr(X\tp Y)$,
	where $\tr(A)$ is the trace of a matrix $A\in \bb{R}^{n\times p}$.
	We use $\norm{\cdot}_2$ and $\norm{\cdot}\ff$ to represent the $2$-norm and the Frobenius norm, respectively. 
	The notations $\mathrm{diag}(A)$ and $\Diag(x)$
	stand for the vector formed by the diagonal entries of the matrix $A$,
	and the diagonal matrix with the entries of $x\in\bb{R}^n$ as its diagonal, respectively. 
	We denote the smallest eigenvalue of a square matrix $A$ by $\lambda_{\mathrm{min}}(A)$, and set $\Phi(M):= \frac{1}{2}(M+M\tp)$. 
	
	Moreover,  we set the metric on the Stiefel manifold as the metric inherited from
	the inner product in $\bb{R}^{n\times p}$, and denote 
		$\ca{T}_X := \{D \in \bb{R}^{n\times p}: \Phi(D\tp X) = 0 \}$ and 
		$\ca{N}_X := \{D \in \bb{R}^{n\times p}: D = X\Lambda, \Lambda = \Lambda\tp \}$.
		Namely, 
	$\ca{T}_X$ and $\ca{N}_X$ represents the tangent space and the normal space of the 
	Stiefel manifold at $X$, respectively.

	Additionally, $\ProjS(X) = UV\tp$ denotes the orthogonal projection to Stiefel manifold, where $X = U\Sigma V\tp$ is the economical SVD of $X$
	with $U\in \ca{S}_{n,p}$, $V\in \ca{S}_{p,p}$ and $\Sigma$
	is a $p\times p$ diagonal matrix with the singular values of
	$X$ on its diagonal.  Moreover, the ball centered at $X$ with radius $\delta$ is defined as $\ca{B}(X, \delta):=  \{ Y \in \bb{R}^{n\times p}: \norm{Y - X}\ff \leq \delta  \}$. 
	Furthermore, $\JA(X)[D]$ refers to the linear transform of $D$ by the linear mapping $\JA(X)$, i.e., $\JA(X)[D] = \lim_{t \to 0} \frac{1}{t} (\A(X+tD) - \A(X))$. Additionally, for any subset $\ca{F} \subseteq \bb{R}^{n\times p}$, we denote $\JA(X)[\ca{F}] := \{ \JA(X)[D]: D \in \ca{F} \}$. 
	Finally, we denote $g(X) := f(\A(X))$, and $\Omega_r = \Big\{ X \in \bb{R}^{n\times p} \,:\, \norm{X^\top X - I_p}\ff \leq r  
	\Big\}$. 
	
	\subsection{Organization}
	The rest of this paper is organized as follows. Section 2 presents several preliminary notations and definitions.  We analyze the relationships between \ref{Prob_Ori} and \ref{NEPen} in Section 3. We show how to implement subgradient methods for \ref{Prob_Ori} by \ref{NEPen} and prove its theoretical convergence directly from existing works in Section 4. In Section 5, we present preliminary numerical experiments to illustrate that \ref{NEPen} enables efficient and direct implementation of various existing unconstrained solvers to solve \ref{Prob_Ori}. We draw a brief conclusion in the last section.
	
	\section{Preliminaries}
	\subsection{Definitions}

	\begin{defin}
		
		Given $X \in \bb{R}^{n\times p}$, the generalized directional derivative of $f$ at $X$  in the direction $D \in \bb{R}^{n\times p}$, denoted by $f^\circ(X, D)$, is defined as 
		\begin{equation*}
			f^\circ(X, D) = \mathop{\lim\sup}_{Y\to X, t \downarrow  0}~ \frac{f(Y +tD) - f(Y)}{t}. 
		\end{equation*}
		Then the generalized gradient or the Clarke subdifferential of $f$ at $X \in \bb{R}^{n\times p}$, denoted by $\partial f(X)$, is defined as 
		\begin{equation*}
			\partial f(X) := \{ W \in \bb{R}^{n\times p} : \inner{W, D} \leq f^\circ (X, D), \text{ for all } D \in \bb{R}^{n\times p} \}.
		\end{equation*}
	\end{defin}

	\begin{defin}
		\label{Defin_regular}
		We say that $f$ is (Clarke) regular at $X \in \bb{R}^{n\times p}$ if for every direction $D$, the one-sided directional derivative 
		\begin{equation*}
			f^\star(X,D) = \lim_{t \downarrow 0} ~\frac{f(X + tD) - f(X)}{t} 
		\end{equation*}
		exists and $f^\star(X, D) = f^\circ(X, D)$.
	\end{defin}

	Next we follow the definition in \cite{yang2014optimality,hosseini2018line} to present the definition for Riemannian subdifferential on the Stiefel manifold. 
	\begin{defin}
		Given $X \in \ca{S}_{n,p}$, the generalized Riemannian directional derivative of $f$ at $X$ in the direction $D \in \ca{T}_X$, denoted by $f^\diamond(X, D)$, is given by 
		\begin{equation*}
			f^\diamond(X, D) = \mathop{\lim\sup}_{ \ca{S}_{n,p} \ni Y\to X, t \downarrow  0}~ \frac{f\left( \ProjS(Y + tD) \right) - f(Y)}{t}. 
		\end{equation*}
		Then the Riemannian subdifferential, denoted as $\partialRie f(X)$, is defined as 
		\begin{equation*}
			\partialRie f(X) := \{ W \in \ca{T}_X : \inner{W, D} \leq f^\diamond (X, D), ~\text{for all } D \in  \ca{T}_X \}.
		\end{equation*}
	\end{defin}

	\begin{defin}
		\label{Defin_Projected_subdifferential}
		Given any $X\in \ca{S}_{n,p}$, the projected subdifferential of $f$ at $X \in \ca{S}_{n,p}$ on the Stiefel manifold is defined as $\partialProj f(X) = \{ W - X\Phi(X\tp W): W \in \partial f(X) \}$. 
	\end{defin}

	Moreover, we define the first-order stationary points of problem \ref{Prob_Ori} as follows.
	\begin{defin}[{\cite[Theorem 6.11]{clarke1990optimization}}]
		\label{Defin_FOSP}
		Given $X \in \ca{S}_{n,p}$, we say $X$ is a first-order stationary point of \ref{Prob_Ori} if and only if $0 \in \partialProj f(X)$. 
	\end{defin}

	\begin{prop}[{\cite[Theorem 5.1]{yang2014optimality}}]
		\label{Prop_relationship_partialR_partialP}
		For any given $X \in \ca{S}_{n,p}$, it holds that  $\partialRie f(X) \subseteq \partialProj f(X)$. 
		Furthermore,  the equality holds when $f$ is Clarke regular. 
	\end{prop}
	
	\begin{rmk}
		It is worth mentioning that Proposition \ref{Prop_relationship_partialR_partialP} implies that any local minimizer $\tX$ of \ref{Prob_Ori} satisfies $0 \in \partialProj f(X)$.  Moreover,  Definition \ref{Defin_FOSP} coincides with the widely used optimality conditions \cite{yang2014optimality} for regular objective functions. Additionally, when $\partial f(X)$ is available, the projected subdifferential $\partialProj f(X)$ is usually much easier to compute than $\partialRie f(X)$. 
	\end{rmk}

	\begin{defin}[\cite{clarke1990optimization,rockafellar2009variational}]
		For any locally Lipschitz continuous function $\psi: \bb{R}^{n\times p} \to \bb{R}$, we say $X \in \bb{R}^{n\times p}$ is a first-order stationary point of function $\psi$ if and only if 
		\begin{equation*}
			0 \in \partial \psi(X). 
		\end{equation*}
	\end{defin}

	Finally, we define the following constants for the theoretical analysis of \ref{NEPen},
	\begin{itemize}
		\item $M_0 := \sup_{X \in \Omega_1  } ~  f(X) - \inf_{X \in \Omega_1 } f(X)$;
		\item $M_1 := \sup_{X \in \Omega_1, W \in \partial f(X) } ~ \norm{W}\ff.$
	\end{itemize}
	It is worth mentioning that both $M_0$ and $M_1$ are independent of the penalty parameter $\beta$.

	\section{Theoretical Properties}
	
	\subsection{Preliminary lemmas}
	In this subsection, we first present the expression for the Jacobian of the nonlinear mapping $\A$ in \eqref{A-map}, which could be directly derived from the expression of $\A$. Here we denote $\JA(X)$ as the Jacobian of the mapping $\A$ at $X$, which can be regarded as a linear mapping from $\bb{R}^{n\times p}$ to $\bb{R}^{n\times p}$. 
	The following proposition illustrates that for any $X \in \Omega_{1}$,  mapping $\A$ to $X$ can cubically reduce the feasibility violation. 
	\begin{prop}
		\label{Prop_ATA}
		For any given $X \in \bb{R}^{n\times p}$, it holds that 
		\begin{equation} \label{eq-ATA}
			\A(X)\tp\A(X) - I_p = \frac{1}{64}
			(X\tp X - I_p)^3 \big( 9 (X\tp X)^2 - 33 (X\tp X) + 64 I_p\big).
		\end{equation} 
		Moreover, for any $X\in \Omega_1$, we have  $\norm{\A(X)\tp\A(X) - I_p}_{\mathrm{F}} \,\leq \, 
		\norm{X\tp X - I_p}_{\mathrm{F}}^3$.
	\end{prop}
	\begin{proof} By using the  economical SVD of $X = U\Sigma V\tp$, we have that
	$$
	\A(X)\tp \A(X) -I_p  = \frac{1}{64} V \Big(
	\Sigma^2 (15 I_p - 10 \Sigma^2 + 3 \Sigma^4)^2 -64 I_p \Big) V^T.
	$$
	By considering the polynomial function $p(x) = x (15-10 x+ 3x^2)^2 - 64$ and noting
	that $p(1) = p'(1) = p''(1) = 0$, we can factorize it as 
	$p(x) = (x-1)^3(9x^2 -33 x + 64).$ From here, we can readily obtain the required
	result in \eqref{eq-ATA}.
	
	Next, for any $X\in \Omega_1$ with the economical SVD of $X = U\Sigma V\tp$, we have
	that $\Sigma^2 \preceq 2 I_p$, and
	\begin{align*}
	&\norm{\A(X)\tp\A(X) - I_p}_{\mathrm{F}} \leq \frac{1}{64}\norm{X\tp X - I_p}_{\mathrm{F}}^3
	\norm{ 9 (X\tp X)^2 - 33 (X\tp X) + 64 I_p}_2
	\\
	 ={}& \frac{1}{64}\norm{X\tp X - I_p}_{\mathrm{F}}^3\; \max_{0\leq x \leq 2}\{
	 |9 x^2 - 33 x + 64|\}
	 =  \norm{X\tp X - I_p}_{\mathrm{F}}^3.
	\end{align*}
	This completes the proof.
	\end{proof}
	
	Next we present the Jacobian of $\A$ in the following proposition. 
	\begin{prop}
		\label{Prop_diff_A}
		For any $X\in \bb{R}^{n\times p}$, it holds that for any $D\in \bb{R}^{n\times p}$
		\begin{equation*}
			\JA(X)[D] = \frac{1}{8}D \left( 15 I_p - 10 X\tp X + 3(X\tp X)^2\right) - X \Phi(X\tp D) + \frac{3}{2}X\Phi(\Phi(X\tp D)(X\tp X - I_p)). 
		\end{equation*}
		Moreover, for any $X \in \ca{S}_{n,p}$, we have
		\begin{equation*}
			\JA(X)[D] = D - X\Phi(D\tp X). 
		\end{equation*}
	\end{prop}
	The statements in Proposition \ref{Prop_diff_A} can be verified by straightforward but tedious calculations, and hence its proof is omitted.

	Based on the formulation of $\JA$, the following lemma illustrates that $\JA(X)$ is a self-adjoint mapping in $\bb{R}^{n\times p}$. 
	\begin{lem}
		For any given $X\in \bb{R}^{n\times p}$, the mapping $\JA(X)$ is self-adjoint. 
	\end{lem}
	
	\begin{proof}
		For any $D,W\in \bb{R}^{n\times p}$, it is easy to verify that 
		\begin{align*}
			& \tr\left(W\tp D \left( {15} I_p - {10}X\tp X + {3} (X\tp X)^2 \right)\right) = \tr\left(D\tp W \left( {15} I_p - {10}X\tp X + {3} (X\tp X)^2 \right)\right),\\
			& \tr\left( W\tp X \Phi(X\tp D) \right) = \tr\left( D\tp X \Phi(X\tp W) \right).
		\end{align*}
		Moreover, through direct calculation, we achieve
		\begin{equation*}
			\begin{aligned}
				&\tr\left(W\tp X\Phi(\Phi(X\tp D)X\tp X)\right)= \tr\left(\Phi(X\tp W)\Phi(X\tp D)X\tp X\right)\\
				={}& \tr\left(X\tp X\Phi(X\tp W)\Phi(X\tp D)\right) =  \tr\left( D\tp X\Phi(\Phi(X\tp W)X\tp X) \right).
			\end{aligned}
		\end{equation*}
		Therefore, we conclude that 
		\begin{equation*}
			\begin{aligned}
				&\inner{\JA(D),W} 
				= \inner{\frac{1}{8}D \left( 15 I_p - 10 X\tp X + 3
				(X\tp X)^2 \right) - \frac{5}{2}X \Phi(X\tp D) +\frac{3}{2}X\Phi(\Phi(X\tp D)X\tp X),W}\\
				={}& \inner{\frac{1}{8}W \left( 15 I_p - 10X\tp X + 3 (X\tp X)^2 \right) - \frac{5}{2}X \Phi(X\tp W) +\frac{3}{2}X\Phi(\Phi(X\tp W)X\tp X),D}\\
				={}& \inner{\JA(W),D}, 
			\end{aligned}
		\end{equation*}
		and this completes the proof.
	\end{proof}

	In the rest of this subsection, we aim to prove that $\partial h$ can be explicitly formulated from $\partial f$. We first present the following theorem, which is a direct corollary from \cite[Theorem 2.3.10]{clarke1990optimization} that characterizes $\partial g(X)$.
	\begin{theo}{\cite[Theorem 2.3.10]{clarke1990optimization}}
		\label{The_Clarke}
		For any given $X \in \bb{R}^{n\times p}$, it holds that 
		\begin{equation*}
			\partial (f\circ \A)(X) \subseteq \JA(X)[\partial f(\A(X))].
		\end{equation*}
		Moreover, when $\A$ maps any neighborhood of $X$ to a  set which is dense in a neighborhood of  $\A(X)$, we have $\partial (f\circ \A)(X) = \JA(X)[\partial f(\A(X))]$.
		
	\end{theo}

	Then following lemma illustrates that $\A$ is a homeomorphism from $\bb{R}^{n\times p}$ to $\bb{R}^{n\times p}$. 
	\begin{lem}
		\label{Le_dense_surjective}
		The mapping  $\A$ in \eqref{A-map} is a homeomorphism from $\bb{R}^{n\times p}$ to $\bb{R}^{n\times p}$. Moreover, for any given $X \in \bb{R}^{n\times p}$ and any $\varepsilon > 0$, there exists a $\delta > 0$ such that $\ca{B}(\A(X), \delta) \subset \A(\ca{B}(X, \varepsilon) )$.
	\end{lem}
	\begin{proof}
		Consider the auxiliary function $\psi(t) = t(\frac{15}{8} - \frac{5}{4}t^2 + \frac{3}{8}t^4 )$. Since $\psi'(t) = \frac{15}{8}(t^2 - 1)^2$, we can conclude that $\psi(t)$ is non-decreasing in $\bb{R}$ and $\psi'(t)$ is positive except for $t = \pm 1$. As a result, $\psi(t)$ is an injection from $\bb{R}$ to $\bb{R}$, and it is easy to verify that $\psi$ is a bijection and  $\psi^{-1}$ is continuous in $\bb{R}$.
		
		Let $X = U\Sigma V\tp$ be the rank-revealing singular value decomposition of $X$, where both $U \in \bb{R}^{n\times p}$ and $V \in \bb{R}^{p\times p}$ are orthogonal matrices, and $\Sigma \in \bb{R}^{p\times p}$ is diagonal. Then from the definition of $\psi$, we can  conclude that $\A(X) = U \psi(\Sigma)V\tp$, and $\A$ is an injection. Moreover,  for any given $Y \in \bb{R}^{n\times p}$ and its singular value decomposition $Y = U_Y \Sigma_Y V_Y\tp$, from the definition of $\psi^{-1}$, it holds that $Y = \A(U_Y \psi^{-1}(\Sigma_Y) V_Y\tp)$.  As a result, $\A$ is a bijection and $\A^{-1}$ is well-defined and continuous in $\bb{R}^{n\times p}$. 
		
		Therefore, $\A$ maps any open set to an open set due to the continuity of $\A^{-1}$. As a result, we can conclude that for any $X \in \bb{R}^{n\times p}$ and any $\varepsilon > 0$,  $\A(\ca{B}(X, \varepsilon) )$ is an open set that contains $\A(X)$. Therefore, there exists $\delta > 0$ such that $\ca{B}(\A(X), \delta) \subset \A(\ca{B}(X, \varepsilon) )$ and thus we complete the proof.

	\end{proof}

	Together with Theorem \ref{The_Clarke}, Lemma \ref{Le_dense_surjective}  shows that  the chain rule holds for $f\circ \A$, as summarized in the following proposition. Therefore, the Clarke subdifferential of $g$ is easy to achieve from $\partial f$ even if $f$ is only assumed to be locally Lipschitz continuous and without any conditions on its regularity. 
	\begin{prop}
		\label{Prop_chain_rule}
		For any given $X \in \bb{R}^{n\times p}$, it holds that $\partial g(X) = \JA(X)[\partial f(\A(X))]$, and $\partial h(X) = \JA(X)[\partial f(\A(X))] + \beta X(X\tp X - I_p)$. 
	\end{prop}
	\begin{proof}
		From Lemma \ref{Le_dense_surjective}, and Theorem \ref{The_Clarke} (Theorem 2.3.10 in \citep{clarke1990optimization}), we  have that the chain rule holds for $f(\A(X))$ and 
		this leads to the desired result. 
	\end{proof}

	\begin{rmk}
		From the definition of $h(X)$ and Lemma \ref{Prop_chain_rule}, we can easily verify that for any $X \in \ca{S}_{n,p}$, it holds that $h(X) = f(X)$ and $\partialProj f(X) = \partial h(X)$.
	\end{rmk}

	\begin{rmk}
		Among all the fifth-order polynomials of $X$ that are constraint dissolving mappings for \ref{Prob_Ori}, the mapping $\A$ in \ref{NEPen} is the unique homeomorphic constraint dissolving mapping. 
		All the other choices of the homeomorphic constraint dissolving mappings are either higher-order polynomials of $X$ or are not polynomials of $X$, hence they are more expensive to compute than the mapping $\A$ in \ref{NEPen}.   More precisely, let 
		\begin{equation}
			\tilde{\A}(X) := X\left( \alpha_0 I_p + \alpha_1 (X\tp X) + \alpha_2(X\tp X)^2 \right).
		\end{equation}
		The formulation of $\tilde{\A}$ captures all the possible choices of the fifth-order polynomials of $X$  that are constraint dissolving mappings for \ref{Prob_Ori}. By Condition \ref{cond:1}, the coefficients should satisfy
		\begin{equation}
			\alpha_0 + \alpha_1 + \alpha_2 = 1, \qquad \text{and} \quad \alpha_0 + 3\alpha_1 + 5\alpha_2 = 0.
		\end{equation}
		Moreover, when $\tilde{A}$ is a homeomorphism over $\bb{R}^{n\times p}$, the functions $\psi(t) = \alpha_0 t + \alpha_1t^3 +\alpha_2t^5$ should be a homeomorphism over $\bb{R}$. Therefore, the only possible choice for is $(\alpha_0, \alpha_1, \alpha_2) = (\frac{15}{8}, -\frac{5}{4}, \frac{3}{8})$, which leads to the formulation of \ref{NEPen}.  
	\end{rmk}

	\subsection{Relationships on stationary points}
	\label{Subsection_equivalence_stationary}
	In this section, we build up the relationships between the first-order stationary points of \ref{Prob_Ori} and \ref{NEPen}. 
	
	\begin{lem}
		\label{Le_JA_JC}
		For any given $X \in \bb{R}^{n\times p}$ and any $W \in \bb{R}^{n\times p}$, it holds that
		\begin{equation*}
			\inner{\JA(X) [W], X(X\tp X - I_p)} = \frac{15}{8} \inner{\Phi(X\tp W), (X\tp X - I_p)^3}. 
		\end{equation*}
	\end{lem}
	\begin{proof}
		From Proposition \ref{Prop_diff_A}, we can obtain after some tedious calculations that
		\begin{equation*}
			\begin{aligned}
				&\inner{\JA(X) [W], X(X\tp X - I_p)}
				=\inner{X\tp \JA(X) [W],(X\tp X - I_p)}\\
				=&\inner{\Phi(X\tp W),\frac{15}{8}((X\tp X)^2-2X\tp X+I_p)(X\tp X - I_p)} \\
				=& \frac{15}{8}\inner{\Phi(X\tp W), (X\tp X - I_p)^3},
			\end{aligned}
		\end{equation*}
		and  the proof is completed.
	\end{proof}
	The following theorem illustrates that any first-order stationary point of \ref{NEPen} is a first-order stationary point of \ref{Prob_Ori}. 
	\begin{theo}
		\label{The_Equivalence_local}
		
		For any given $\beta > 0$, $r \in \left(0,  \frac{\beta}{2\beta + 8M_1} \right)$, and any $X \in \Omega_{r}$, it holds that 
		\begin{equation*}
			\mathrm{dist}(0, \partial h(X)) \geq \frac{\beta}{4} \norm{X\tp X - I_p}\ff. 
		\end{equation*}
		Furthermore, when $X^* \in \Omega_{r}$ is a first-order stationary point of \ref{NEPen}, then $X^*$ is a first-order stationary point of \ref{Prob_Ori}. 
	\end{theo}
	\begin{proof}
		Suppose $X$ is not feasible, then $\norm{ {X}\tp X - I_p }\ff > 0$.  
		Since $X\in \Omega_r$,  it holds that  $\norm{ {X}\tp X - I_p }\ff \leq r < 1$ and 
		\begin{equation*}
			\sigma_{\min}({X}\tp X) \geq 1 - \sigma_{\max}({X}\tp X - I_p) \geq 1 - \norm{{X}\tp X - I_p}\ff. 
		\end{equation*}
		Then for any $D \in \partial h(X)$, by Proposition \ref{Prop_chain_rule}
		 there exists $W \in \partial f(\A(X))$ such that $D = J_{\A}(X)[W] + \beta  X({X}\tp X - I_p)$. Moreover, by Proposition \ref{Prop_ATA}, $\A(X)\in \Omega_{r^3} \subset \Omega_1$, and
		 hence $\norm{W}\ff \leq M_1$. In addition, since $\norm{X}_2 \leq 2$, we get
		 $\norm{\Phi(X\tp W)}\ff \leq \norm{X}_2 \norm{W}\ff \leq 2 M_1.$ 
		Now we can conclude that
		\begin{equation*}
			\begin{aligned}
				&2\norm{D}\ff \norm{X\tp X - I_p}\ff \geq \norm{D}\ff \norm{X}_2\norm{X\tp X - I_p}\ff  
				\geq \inner{D, {X}({X}\tp {X} - I_p)}\\
				\geq {}&   \beta  \inner{{X}({X}\tp {X} - I_p),{X}({X}\tp {X} - I_p) } - \left| \inner{ J_{\A}(X)[W],  {X}({X}\tp {X} - I_p)} \right| \\
				\geq{}& \beta \sigma_{\min}({X}\tp X ) \norm{{X}\tp X - I_p}\fs 
				- \frac{15}{8} \left|\inner{\Phi({X}\tp W), (X\tp X - I_p)^3}\right| \\
				\geq{}& \beta \sigma_{\min}({X}\tp X ) \norm{{X}\tp X - I_p}\fs - \frac{15}{8} \norm{\Phi({X}\tp W)}_2 \norm{X\tp X - I_p}\ff^3\\
				\geq{}& \beta\norm{{X}\tp X - I_p}\fs - \left(\frac{15 M_1}{4} + \beta\right)\norm{{X}\tp X - I_p}\ff^{3}\\
				\geq {}&\beta\norm{{X}\tp X - I_p}\fs \left( 1 - \frac{\beta + 4M_1}{\beta} \norm{{X}\tp X - I_p}\ff \right) 
				\geq \frac{\beta}{2} \norm{{X}\tp X - I_p}\fs.
			\end{aligned}
		\end{equation*}
		Therefore, it holds that $\norm{D}\ff \geq \frac{\beta}{4} \norm{{X}\tp X - I_p}\ff$. 
		By the arbitrariness of $D \in \partial h(X)$, we achieve that $\mathrm{dist}(0, \partial h(X)) \geq \frac{\beta}{4} \norm{X\tp X - I_p}\ff$. 
		Furthermore, when $X^* \in \Omega_{r}$ is a first-order stationary point of \ref{NEPen}, then 
		\begin{equation*}
			0 = \mathrm{dist}(0, \partial h(X^*)) \geq \frac{\beta}{4} \norm{{X^*}\tp {X^*} - I_p}\ff \geq 0,
		\end{equation*}
		which implies that  ${X^*}\tp X^* = I_p$ and this completes the proof. 
	\end{proof}

	\subsection{Estimating stationarity}
	Directly applying unconstrained optimization approaches to solve \ref{NEPen} usually yields an infeasible sequence and terminates at an infeasible point. However, sometimes we expect mild accuracy in the substationarity, but pursue high accuracy in the feasibility in minimizing \ref{NEPen}. To this end, we impose an orthonormalization post-process after obtaining the solution $X$  by applying unconstrained approaches to solve \ref{NEPen}. Namely, 
	\begin{eqnarray}\label{orth}
		\Xo := \ProjS(X),
	\end{eqnarray}
	where $\ProjS: \bb{R}^{n\times p} \to \ca{S}_{n,p}$ is the projection on Stiefel manifold defined in Section 1.4. We first estimate the distance between $X$ and $\ProjS(X)$ in the following lemma. 
	\begin{lem} \label{lem-proj}
		For any given $X \in \bb{R}^{n\times p}$, it holds that
		\begin{equation*}
			\norm{X - \ProjS(X)}\ff \leq \norm{X\tp X - I_p}\ff. 
		\end{equation*}
	\end{lem}
	\begin{proof}
		By the economical SVD of $X$, that is, $X = U \Sigma V\tp$ with
		$\Sigma = {\rm diag}(\sigma)$, we can achieve that 
		\begin{equation*}
			\norm{X - \ProjS(X)}\ff = \norm{\Sigma - I_p}\ff \leq \norm{\Sigma^2 - I_p}\ff = \norm{X\tp X - I_p}\ff. 
		\end{equation*}
		The inequality above holds simply because $|x-1| \leq |x+1||x-1|$ 
		for any $x\geq 0$. 
	\end{proof}
	
	Then the following lemma estimates the distance between $\A(X)$ and $\ProjS(X)$. 
	\begin{lem} \label{lem-A-proj}
		For any given $X \in \Omega_{1/2}$, it holds that
		\begin{equation} \label{eq-A-proj}
			\norm{\A(X) - \ProjS(X)}\ff \leq 4\norm{X\tp X - I_p}\ff^3.
		\end{equation}
		
		Moreover 
		\begin{eqnarray} \label{eq-fA-fproj}
		 | f(\A(X)) - f(\ProjS(X))| \;\leq \; M_1 \norm{\A(X) - \ProjS(X)}\ff.
		\end{eqnarray}
	\end{lem}
	\begin{proof}
		By the economical SVD of $X$, that is, $X = U \Sigma V\tp$, we obtain that
		$$
		  \A(X) - \ProjS(X) = \frac{1}{8} U\Big( 3 \Sigma^5 - 10 \Sigma^3 + 15 \Sigma - 8 I_p
		  \Big) V^T.
		$$
		Note that $X\in \Omega_{1/2}$ implies that $0\prec \Sigma \preceq \frac{3}{2} I_p$.
		Next, we have that
		\begin{equation*}
			\begin{aligned}
				&\norm{\A(X) - \ProjS(X)}\ff = \frac{1}{8}\norm{3\Sigma^5 - 10 \Sigma^3 
				+ 15\Sigma - 8 I_p }\ff \\
				={}& \frac{1}{8}\norm{\left( \Sigma - I_p \right)^3 \left( 3 \Sigma^2 + 9\Sigma + 8I_p   \right)}\ff 
				\leq  \frac{1}{8}
				\norm{ 3 \Sigma^2 + 9\Sigma + 8I_p   }_2 \norm{\left( \Sigma - I_p \right)^3}\ff
				 \\
				\leq {}&  \left(\frac{1}{8} \max_{0\leq x \leq 3/2} \{
				 |3 x^2 + 9 x + 8|\}\right) \norm{\Sigma - I_p}\ff^3
				\leq 4\norm{\Sigma^2 - I_p}\ff^3 = 4\norm{X\tp X - I_p}\ff^3. 
			\end{aligned}
		\end{equation*}
		Note that we used Lemma \ref{lem-proj} in the last inequality.

		Next we prove \eqref{eq-fA-fproj}. 
		Since $f$ is locally Lipschitz continuous, by \cite[Theorem 2.3.7]{clarke1990optimization}, 
		there exists $t\in (0,1)$ such that 
		$f(\A(X)) - f(\ProjS(X)) \in \{ \langle W, \, \A(X) - \ProjS(X)\rangle : W \in \partial f(Z_t)\}$, where $Z_t =   t \A(X) + (1-t) \ProjS(X)$. Thus
		$$
		 |f(\A(X)) - f(\ProjS(X))| \leq  \norm{\A(X) - \ProjS(X)}\ff \sup_{W\in \partial f(Z_t)} \norm{W}\ff
		$$ 
		By some tedious calculations and using the result in \eqref{eq-A-proj}, we can show that
		$Z_t \in \Omega_1$, and hence $\sup_{W\in \partial f(Z_t)} \norm{W}\ff \leq M_1$.
		This completes the proof.
		
	\end{proof}

	The following proposition illustrates that the post-processing procedure \eqref{orth}  can further reduce the function value simultaneously if the current point is close to the Stiefel manifold. 
	
	\begin{prop}
		\label{Le_ortho_fval}
		
		For any given 
		$X \in \Omega_{1/2}$,  it holds that
		\begin{equation*}
			h(\ProjS(X)) \leq h(X) - \left( \frac{\beta}{4} -  4M_1\norm{X\tp X - I_p}\ff \right)\norm{X\tp X - I_p}\fs.
		\end{equation*}
	\end{prop}	
	\begin{proof}
		\begin{equation*}
			\begin{aligned}
				&h(X) - h(\ProjS(X)) = f(\A(X)) - f(\ProjS(X)) +  \frac{\beta}{4}\norm{X\tp X - I_p}\fs \\
				\overset{(i)}{\geq}{}& -M_1 \norm{\A(X) - \ProjS(X)}\ff + \frac{\beta}{4}\norm{X\tp X - I_p}\fs  \\
				\geq{}& \left( \frac{\beta}{4} -  4M_1 \norm{X\tp X - I_p}\ff \right)\norm{X\tp X - I_p}\fs.
			\end{aligned}
		\end{equation*}
		Here $(i)$ follows from \eqref{eq-fA-fproj} in Lemma \ref{lem-A-proj}.
	\end{proof}

	\begin{theo}
		\label{The_Equivalence_local_minimizer}
		
		For any given $\beta > 0$, $r \in \left(0,  \frac{\beta}{2\beta + 16M_1} \right)$, and any $X^* \in \Omega_{r}$, suppose  $X^*$ is a local minimizer of \ref{NEPen}, then  $X^*$ is a local minimizer of \ref{Prob_Ori}. Moreover,  any local minimizer $X^*$ of \ref{Prob_Ori} is a local minimizer of \ref{NEPen}. 
	\end{theo}
	\begin{proof}
		For any given $\beta > 0$, $r \in \left(0,  \frac{\beta}{2\beta + 16M_1} \right)$, and any $X^* \in \Omega_{r}$, suppose $X^* \in \Omega_{r}$ is a local minimizer of \ref{NEPen}. Then $0 \in \partial h(X^*)$ and Theorem \ref{The_Equivalence_local} further implies that $X^* \in \ca{S}_{n,p}$. Notice that $h(X) = f(X)$ holds for any $X \in \ca{S}_{n,p}$, we immediately obtain that $X^*$ is a local minimizer of \ref{Prob_Ori}. 
		
		On the other hand, when $X^* \in \ca{S}_{n,p}$ is a local minimizer of \ref{Prob_Ori},  then there exists a constant $\gamma \in \left(0, \min\{\frac{1}{2},\frac{\beta}{2\beta + 16M_1} \} \right) $ such that $f(Z) \geq f(X^*)$ holds for any $Z \in \ca{S}_{n,p}$ satisfying $\norm{Z - X^*}\ff \leq \gamma$. 
		As $X^*$ is feasible, the set 
		\begin{equation*}
			\left\{Y \in \bb{R}^{n\times p}: \norm{Y-X^*}\ff \leq \frac{\gamma}{2}\right\} \cap \left\{Y \in \bb{R}^{n\times p}: \norm{Y\tp Y - I_p}\ff \leq \frac{\gamma}{2}\right\}
		\end{equation*}
		is a neighborhood of $X^*$. 
		Then for any $Y \in \bb{R}^{n\times p}$ such that  $\norm{Y - X^*}\ff \leq \frac{\gamma}{2}$ and $\norm{Y\tp Y - I_p}\ff \leq \frac{\gamma}{2}$, we first have that 
		\begin{equation*}
			\norm{\ProjS(Y) - X^*}\ff \leq \norm{\ProjS(Y) - Y}\ff + \norm{Y-X^*}\ff \leq \frac{\gamma}{2} + \frac{\gamma}{2} =  \gamma.
		\end{equation*}
		Then from Proposition \ref{Le_ortho_fval} we can conclude that 
		\begin{equation*}
			\begin{aligned}
				h(Y) - h(X^*) ={}& h(Y) - h(\ProjS(Y)) + h(\ProjS(Y)) - h(X^*) \\
				\geq{}& \left( \frac{\beta}{4} -  4M_1 \norm{Y\tp Y - I_p}\ff \right) \norm{Y\tp Y - I_p}\fs \geq 0,
			\end{aligned}
		\end{equation*}
		and complete the proof. 
	\end{proof}

	\begin{rmk}
		For the following optimization problem that minimizes a nonsmooth objective function over product of multiple Stiefel manifolds and Euclidean spaces,
			\begin{equation}
				\label{Eq_Prob_Multi_Ori}
				\begin{aligned}
					\min_{Y_i \in \bb{R}^{n_i\times p_i}, Z \in \bb{R}^l} \quad& f(Y_1,...,Y_N, Z)\\
					\text{s. t.} \quad & Y_i\tp Y_i = I_{p_i}, ~\text{for~} i = 1,...,N,
				\end{aligned}
			\end{equation}
			the corresponding constraint dissolving function can be formulated as
			\begin{equation}
				\label{Eq_Prob_Multi_Pen}
				f(\A_1(Y_1),...,\A_N(Y_N), Z) + \frac{\beta}{4} \left( \sum_{i = 1}^N \norm{Y_i\tp Y_i - I_{p_i}}\fs \right),
			\end{equation}
			where $\A_i(Y_i) := \frac{1}{8} Y_i \left( 15 I_{p_i} - 10 Y_i\tp  Y_i + 3(Y_i\tp Y_i)^2 \right)$. 
			The relationship between \eqref{Eq_Prob_Multi_Ori} and \eqref{Eq_Prob_Multi_Pen}  can be established in the same way as in the proofs in Section 3.

		Moreover, the exactness of \ref{NEPen} can be further extended to optimization problems over the generalized Stiefel manifold,
		\begin{equation}
			\label{Prob_Gen}
			\begin{aligned}
				\min_{X \in \bb{R}^{n\times p}} \quad& f(X)\\
				\text{s. t.} \quad & X\tp B X = I_p,
			\end{aligned}
		\end{equation}
		where $B \in \bb{R}^{n\times n}$ is a symmetric positive definite matrix.  For \eqref{Prob_Gen}, we can consider the following exact penalty function $\hsharp$,
		\begin{equation}
			\label{penalty_function_general}
			\hsharp(X) := f(\A^\sharp(X)) + \frac{\beta}{4} \norm{X\tp BX - I_p}\fs, 
		\end{equation} 
		where 
		\begin{equation}
			\A^\sharp(X):= \frac{1}{8}X \left( 15 I_p - 10 X\tp B X + 3 (X\tp BX)^2 \right).
		\end{equation}
		Similar to the proof of Lemma \ref{Le_dense_surjective}, it is easy to verify that $\A^{\sharp}$ is a homeomorphism from $\bb{R}^{n\times p}$ to $\bb{R}^{n\times p}$. Therefore, we can the prove the exactness of \eqref{penalty_function_general} together with the results in Section 2 and Section 3 through the same approach as the proofs in this paper.
	\end{rmk}

	\section{Subgradient Methods}

		In this section, we discuss how to develop subgradient methods for \ref{Prob_Ori} and establish its convergence properties. Subgradient method and its variants play important roles in minimizing nonsmooth functions that are not regular, in particular, not prox-friendly.  Recently, \cite{davis2020stochastic} shows the global convergence for subgradient methods in minimizing those functions that {\it admit a chain rule}, which is defined later in Definition \ref{Defin_Chain_rule}. Such properties are essential for the so-called ``descent condition'' \cite[Assumption B]{davis2020stochastic} for subgradient methods, hence are fundamental in establishing their convergence properties on solving nonconvex nonsmooth optimization, as illustrated in various existing works \cite{davis2020stochastic,bolte2021conservative,castera2021inertial}. Without assuming such property, as far as we know, there is no technique to analyze these theoretical properties of subgradient methods in such general situations. Therefore, we first introduce the concept of functions that admit a chain rule in Section \ref{Subsection_Whitney}. Moreover,  we show that functions satisfying such property can cover all the problems considered in this paper, and remarkably \ref{NEPen} admits a chain rule whenever the objective function $f$ admits a chain rule.  Furthermore, we propose a stochastic proximal subgradient method for \ref{Prob_Ori} and analyze its convergence properties based on \ref{NEPen}.

	\begin{rmk}
		By transforming \ref{Prob_Ori} into \ref{NEPen}, various existing approaches in unconstrained nonsmooth nonconvex optimization can be directly implemented to solve \ref{Prob_Ori}. These approaches include gradient sampling methods  \cite{burke2005robust,kiwiel2007Convergence,davis2021gradient}, normalized subgradient methods \cite{zhang2020complexity}, nonsmooth second-order approaches \cite{milzarek2019stochastic,asl2021analysis,yang2021stochastic}, etc. Furthermore, the convergence properties of those approaches, including the global convergence and iteration complexity, directly follow the  existing related works \cite{davis2019stochastic,zhang2020complexity,davis2021proximal,davis2021gradient,gratton2020algorithm}. 
		
		Moreover, as illustrated in Theorem \ref{The_Equivalence_local} and Theorem \ref{The_Equivalence_local_minimizer}, for any $\beta > 0$, \ref{Prob_Ori} and \ref{NEPen} have the same first-order stationary points and local minimizers in $\Omega_{r}$ with any $r \in \left(0,  \frac{\beta}{2\beta + 8M_1} \right)$. Therefore,  for any $\beta > 0$ in \ref{NEPen}, whenever these employed algorithms find a first-order stationary point $\tilde{X}$ for \ref{NEPen} which satisfies $\norm{\tilde{X}\tp \tilde{X} - I_p}\ff < \frac{\beta}{2\beta + 8M_1}$, the results presented in Theorem \ref{The_Equivalence_local} and Theorem \ref{The_Equivalence_local_minimizer} illustrate that $\tilde{X}$ is a first-order stationary point of \ref{Prob_Ori}. 
	\end{rmk}

	\subsection{Preliminaries}
	\label{Subsection_Whitney}
	
	In this subsection, we first present the concept of the functions that admit a chain rule. Then we introduce the concept of Whitney stratifiable functions, which plays an important role in nonsmooth analysis and optimization \cite{bolte2007clarke}. As illustrated in \cite{davis2020stochastic}, all the Whitney stratifiable functions and Clarke regular functions admit a chain rule. Furthermore, we provide the definition for several important classes of functions, including semi-algebraic functions, semi-analytic functions, and functions that are definable in an $o$-minimal structure, all of which are contained in the class of Whitney stratifiable functions.


	\begin{defin}
		An absolutely continuous curve is a continuous function $\gamma: [0,1] \to \bb{R}^{n\times p}$ whose derivative exists almost everywhere in $\bb{R}$, and $\gamma(t) - \gamma(0)$ is Lebesgue integral of $\gamma'$ between $0$ and $t$ for all $t \in \bb{R}$, i.e., 
		\begin{equation}
			\gamma(t) = \gamma(0) + \int_{0}^t \gamma'(\tau) \mathrm{d} \tau, \qquad \text{for all $t \in \bb{R}$}.
		\end{equation}
	\end{defin}

	\begin{defin}[Deinfition 5.1 in \cite{davis2020stochastic}]
		\label{Defin_Chain_rule}
		We say a locally Lipschitz continuous function $f$ admits a chain rule if for any absolutely continuous curve $\gamma: \bb{R}_+ \to \bb{R}^{n\times p}$, the equality
		\begin{equation}
			(f\circ \gamma)'(t) = \inner{\partial f(\gamma(t)), \gamma'(t)}
		\end{equation}
		holds for a.e. $t \geq 0$. 
	\end{defin}

	For various applications of \ref{Prob_Ori} arising from machine learning, the objective function $f$ can be the composition and summation of several nonsmooth functions. Therefore, the exact Clarke subdifferential of $f$ is usually extremely difficult to compute in practice in these scenarios, since its sum rule and chain rule fail for general nonsmooth functions. Therefore, the concept of the conservative field \cite{bolte2021conservative} is proposed as a generalization of the Clarke subdifferential for the functions that admit a chain rule. The conservative field is capable of keeping the chain rule and sum rule, hence offers great convenience for studying the stationarity and designing subgradient-based algorithms for unconstrained nonsmooth optimization. We present a brief introduction on conservative field below, and interested readers could refer to \cite{bolte2020mathematical,bolte2021conservative,castera2021inertial,bolte2022nonsmooth} for detailed descriptions and applications of conservative field.
	
	\begin{defin}
		\label{Defin_conservative_field}
		Let $\ca{D}$ be a set-valued mapping from $\bb{R}^{n\times p}$ to subsets of $\bb{R}^{n\times p}$. We call $\ca{D}$ as a conservative field whenever it has closed graph, nonempty compact values, and for any absolutely continuous curve $\gamma: [0,1] \to \bb{R}^{n\times p} $ satisfying $\gamma(0) = \gamma(1)$, we have
		\begin{equation}
			\label{Eq_Defin_Conservative_mappping}
			\int_{0}^1 \max_{v \in \ca{D}(\gamma(t)) } \inner{\gamma'(t), v} \mathrm{d}t = 0, 
		\end{equation}
		where the integral is understood in the Lebesgue sense. 
	\end{defin}

	\begin{defin}
		\label{Defin_conservative_field_path_int}
		Let $\ca{D}$ be a conservative field in $\bb{R}^{n\times p}$. For any given $X_0 \in \bb{R}^{n\times p}$, we can define the function
		\begin{equation}
			\label{Eq_Defin_CF}
			\begin{aligned}
				f(X) = &{} f(X_0) + \int_{0}^1 \max_{v \in \ca{D}(\gamma(t)) } \inner{\gamma'(t), v} \mathrm{d}t
				= f(X_0) + \int_{0}^1 \min_{v \in \ca{D}(\gamma(t)) } \inner{\gamma'(t), v} \mathrm{d}t
			\end{aligned}
		\end{equation} 
		for any absolutely continuous curve $\gamma$ that satisfies $\gamma(0) = X_0$ and $\gamma(1) = X$.  Then $f$ is called a potential function for $\ca{D}$, and we also say $\ca{D}$ admits $f$ as its potential function, or that $\ca{D}$ is a conservative field for $f$. 
	\end{defin}

		\begin{rmk}
			Note that the choice of the conservative field $\D_f$ in Algorithm \ref{Alg:subgradient_sto} depends on how we compute the subdifferential for $f$. When we apply the automatic differentiation algorithms that are widely used in PyTorch and TensorFlow,  the results may not lies in the Clarke subdifferential of $f$. For these cases, $\D_f$ should be chosen as the conservative field that capture all the possible outputs of an automatic differentiation algorithm, as discussed in \cite{bolte2020mathematical,bolte2021conservative}. 
		\end{rmk}

	\begin{lem}[\cite{bolte2021conservative}]
		\label{Le_path_int}
		Suppose $f$ admits a chain rule, then for any absolutely continuous curve $\gamma: [0,1] \to \bb{R}^{n\times p}$, it holds that 
		\begin{equation}
			f(\gamma(1)) - f(\gamma(0)) = \int_{0}^{1} \max_{V \in \partial f(\gamma(t))} \inner{V, \gamma'(t)} \mathrm{d} t = \int_{0}^{1} \min_{V \in \partial f(\gamma(t))} \inner{V, \gamma'(t)} \mathrm{d} t. 
		\end{equation}
	\end{lem}

	It is worth mentioning that the functions that admit chain rules are general enough for a great number of real-world applications.  As shown in \cite{davis2020stochastic}, any Clarke regular function admits a chain rule. Moreover, another import function class is the definable functions.  As illustrated in \cite{van1996geometric}, any function that is definable in an $o$-minimal structure admits a Whitney $\ca{C}^p$- stratification for any $p \geq 1$, and hence admits a chain rule \cite{davis2020stochastic}. It straightforwardly holds from Tarski–Seidenberg theorem \cite{bierstone1988semianalytic} that any semi-algebraic function is definable. 
	Moreover, beyond semialgebraicity, \cite{wilkie1996model} shows that there is an o-minimal structure that simultaneously contains both the graph of the exponential function and all semi-algebraic sets. As a result, various common activation functions and loss functions, such as sigmoid, softplus, RELU, $\ell_1$-loss, MSE loss, hinge loss, logistic loss and cross-entropy loss, are all definable. Additionally, as shown in \cite{davis2020stochastic,bolte2021conservative}, any composition of two definable functions is definable, which further illustrates that the loss function of a neural network built from definable activation functions is definable and thus admits a chain rule.

	Therefore, the objective functions for the applications mentioned in \cite{absil2017collection,chen2018proximal,lerman2018overview,tsakiris2018dual,chen2019manifold,xiao2020l21,xiao2021penalty} are definable. As a result, we can develop subgradient methods to minimize those functions and prove their convergence properties based on the framework presented in \cite{davis2020stochastic}. To develop subgradient methods for \ref{Prob_Ori}, we make the following additional assumptions on $f$ in the rest of this paper. 
	\begin{assumpt}
		\label{Assumption_f}
		For the objective function $f$ in \ref{Prob_Ori}, we assume,
		\begin{enumerate}
			\item The objective function $f$ in \ref{Prob_Ori} is path differentiable;
			\item The set $\{ f(X) : 0 \in \partialProj f(X),~ X \in \ca{S}_{n,p} \}$ has empty interior in $\bb{R}$. 
		\end{enumerate}
		
	\end{assumpt}
	It is worth mentioning that the mapping $\A$ is semi-algebraic, then it follows from \cite{bolte2021conservative} that \ref{NEPen} admits a chain rule whenever the objective function $f$ in \ref{Prob_Ori} admits a chain rule. Moreover, Assumption \ref{Assumption_f}(2) asserts a very weak version of the nonsmooth Sard's theorem \cite{bolte2007lojasiewicz} and plays an important role in establishing the convergence properties of the stochastic subgradient methods, as illustrated in \cite{davis2020stochastic,bolte2021conservative}. The following lemma states that Assumption \ref{Assumption_Subgradient_sto}(2) holds whenever $f$ is definable. Therefore, we can conclude that Assumption \ref{Assumption_f} is mild in practice. 
	\begin{lem}
		Suppose $f$ is definable, then the set $\{ f(X) : 0 \in \partialProj f(X),~ X \in \ca{S}_{n,p} \}$ is finite. 
	\end{lem}
	\begin{proof}
		Notice that both the constraints $X\tp X - I_p$ and the constraint dissolving mapping $\A$ are semi-algebraic, hence they are definable. Therefore, we can conclude that $h$ is definable, as $f$ is definable since  the linear combination and composition of definable functions are definable \cite{bolte2021conservative}. 
		
		Then it holds from \cite{bolte2021conservative} that the critical value for $h$ is finite, i.e., the set $\{ h(X) : 0 \in \partial h(X),~ X \in \bb{R}^{n\times p}\} $ is finite. Notice that Theorem \ref{The_Equivalence_local} illustrates that for any $X \in \ca{S}_{n,p}$, $0 \in \partialProj f(X)$ if and only if $0 \in \partial h(X)$, then we can conclude that 
		\begin{equation*}
			\begin{aligned}
				\{ f(X) : 0 \in \partialProj f(X),~ X \in \ca{S}_{n,p} \} 
				={}& \{ h(X) : 0 \in \partial h(X),~ X \in \ca{S}_{n,p} \}\\ \subseteq{}& \{ h(X) : 0 \in \partial h(X),~ X \in \bb{R}^{n\times p}\},
			\end{aligned}
		\end{equation*}
		which implies that the set $\{ f(X) : 0 \in \partialProj f(X),~ X \in \ca{S}_{n,p} \}$ is finite and the proof is completed. 
	\end{proof}

	Furthermore, when applying subgradient-based algorithms to minimize the penalty function, we usually need to choose $\Dk$ to approximate $\partial g(\Xk)$  at every iterate $\Xk$. The direction $D_k$ is directly chosen from $\partial g(\Xk)$ in deterministic subgradient methods, or with noise in stochastic settings.   
	As illustrated in the following lemma, when we choose the initial point in $\Omega_{1/6}$, the sequence generated by a class of subgradient-based methods for NCDF are uniformly restricted in $\Omega_{1/6}$.  Therefore, in establishing the convergence properties for these subgradient-based methods, the assumptions that $h$ is bounded below in $\bb{R}^{n\times p}$ could be waived.  
		\begin{lem}
		\label{Le_iter_bounded_l2}
		Suppose $Y_0$ satisfies $Y_0 \in \Omega_{1/6}$ and $\Yk$ is generated by 
		\begin{equation*}
			\Ykp = \Yk - \etak (\Dk + \beta \Yk (\Yk \tp \Yk - I_p)),
		\end{equation*}
		where $\Dk \in \bb{R}^{n\times p}$, and there exists a constant $M_2$ such that  
		$\norm{\Dk}\ff \leq M_2$
		holds for any $k \geq 0$. Then when $\beta \geq 60M_2$ and $\etak \leq \frac{1}{2\beta}$, $\Yk \in \Omega_{1/6}$ holds for any $k \geq 0$. 
	\end{lem}
	\begin{proof}
		Note that when $\norm{\Yk\tp \Yk - I_p}\ff \leq \frac{1}{6}$, we have that $5/6I_p \preceq \Yk\tp \Yk \preceq 7/6 I_p$, and hence $\norm{\Yk}_2 \leq 7/6.$
		By the definition of $\Yk$, let $\hat{D}_k := \Dk + \beta \Yk (\Yk \tp \Yk - I_p)$. 
		We have that 
		\begin{equation*}
			\begin{aligned}
				&\norm{\Ykp\tp \Ykp - I_p}\ff 
				= \norm{\Yk\tp \Yk - I_p - \etak \hat{D}_k\tp \Yk - \etak  \Yk\tp \hat{D}_k
					+ \etak^2 \hat{D}_k\tp\hat{D}_k  }\ff \\
				\leq{}& 
				\norm{\Yk\tp \Yk - I_p - 2\etak \beta \Yk\tp \Yk(\Yk\tp \Yk - I_p) + \etak^2 \beta^2 \Yk\tp\Yk (\Yk \tp \Yk - I_p)^2}\ff \\
				&+ \etak\norm{ -\Dk\tp \Yk - \Yk\tp \Dk + \etak\beta  (\Yk\tp \Yk - I_p) \Yk\tp\Dk
					+ \etak\beta \Dk \tp \Yk(\Yk\tp \Yk - I_p) + \etak \Dk\tp\Dk }\ff\\
				\leq{}&		\norm{ (\Yk\tp \Yk - I_p)}\ff \norm{ (I_p - \etak\beta\Yk\tp \Yk)^2 -\etak^2
					\beta^2\Yk\tp\Yk }_2
				+ \etak\Big( 3M_2  + \etak M_2^2\Big) \\
				\leq{}& \left(1 - \frac{5}{6}\etak\beta \right) \norm{\Yk\tp \Yk - I_p}\ff + 3\etak M_2 + \etak^2 M_2^2.
			\end{aligned}
		\end{equation*}
		As a result, when $\norm{\Yk\tp \Yk -I_p}\ff \geq \frac{1}{12}$, 
		\begin{equation*}
			\norm{\Ykp\tp \Ykp - I_p}\ff - \norm{\Yk\tp \Yk - I_p}\ff \leq -\frac{5\etak\beta}{72} +  3\etak M_2 + \etak^2 M_2^2 \leq 0.
		\end{equation*}
		Moreover, when $\norm{\Yk\tp \Yk -I_p}\ff \leq \frac{1}{12}$, $\norm{\Ykp\tp \Ykp - I_p}\ff \leq \norm{\Yk\tp \Yk - I_p}\ff + 3\etak M_2 + \etak^2 M_2^2 \leq \frac{1}{6}$. 
		Then, we have  $\norm{\Ykp\tp \Ykp - I_p}\ff \leq \frac{1}{6}$. Therefore, $\norm{\Yk\tp \Yk - I_p}\ff \leq \frac{1}{6}$ holds for $k = 0,1,2,...$ by induction. 
	\end{proof}

	\subsection{Convergence properties of stochastic subgradient method}
	\label{Subsection_4_subgradient_method}

	In this subsection, we consider developing subgradient methods for solving \ref{Prob_Ori} by applying existing Euclidean subgradient methods to solve \ref{NEPen}. 
	We first present a stochastic subgradient method in Algorithm \ref{Alg:subgradient_sto}.

	\begin{algorithm}[htbp]
		\begin{algorithmic}[1]   
			\Require Function $f$ and penalty parameter $\beta$, conservative field $\D_f$ that admits $f$ as its potential function;
			\State Choose an initial guess $X_0 \in \Omega_{1/6}$, set $k=0$;
			\While{not terminated}
			\State Compute $D_k $ as an approximated evaluation for $ \JA(\Xk)[\D_f(\Xk)]$ ;
			\State $\Xkp = \Xk - \eta_k\left( D_k + \beta\Xk (\Xk\tp \Xk - I_p)\right) $;
			\State $k = k+1$;
			\EndWhile
			\State Return $X_k$.
		\end{algorithmic}  
		\caption{A framework of subgradient methods for solving \ref{Prob_Ori}.}  
		\label{Alg:subgradient_sto}
	\end{algorithm}

	\begin{cond}
		
		\label{Assumption_Subgradient_sto} 
		In Algorithm \ref{Alg:subgradient_sto}, we make the following assumptions.  
		\begin{enumerate}
			\item For any $k \geq 0$, we have $D_k = G_k + E_k$ such that $\lim\limits_{N \to \infty} \sum\limits_{k = 0}^{N} \eta_k E_k $ exists. Moreover, for any subsequence $\{X_{k_j}\}$ that converges to $\tilde{X}$, it holds that 
			\begin{equation*}
				\lim\limits_{j\to \infty} \mathrm{dist}\left( \frac{1}{N}\sum_{j = 1}^N G_{k_j}, \JA(\tilde{X})  [ \D_f(\A(\tilde{X})) ]\right) = 0. 
			\end{equation*}
			\item There exists a constant $\tilde{M}$ such that $\sup_{k \geq 0} \norm{D_k}\ff \leq \tilde{M}$ and $\sup_{X \in \Omega,~ D \in \D_f(X)  } \norm{D} \leq \hat{M}$, and we choose $\beta \geq \max\{ 16M_1, 60\tilde{M}, 16\hat{M} \}$.
			\item  The sequence $\{\eta_k\}$ satisfies 
			\begin{equation*}
				\eta_k > 0, \quad \sum_{k = 0}^{\infty} \eta_k = \infty, \quad \sum_{k = 0}^{\infty} \eta_k^2 < \infty, \quad \sup_{k \geq 0} ~\etak \leq \frac{1}{2\beta}. 
			\end{equation*} 
			\item The set $\{f(X): X \in \ca{S}_{n,p}, 0 \in \JA(X)[\D_f(X)]\}$ has empty interior in $\bb{R}$. 
		\end{enumerate}
	\end{cond}

	In the following, we analyze the convergence properties of Algorithm 1 and show that these subgradient methods are globally convergent based on the framework proposed by \cite{davis2020stochastic}. 
	
	\begin{theo}
		\label{Theo_subgradient_convergence_base}
		Suppose Assumption \ref{Assumption_f} holds and Algorithm \ref{Alg:subgradient_sto} satisfies Condition \ref{Assumption_Subgradient_sto}. Then $\{h(\Xk)\}$ converges and every cluster point $X^*$ of $\{\Xk\}$ lies in the set $\{ X \in \ca{S}_{n,p}: 0 \in \JA(X)[\D_f(X)]  \}$. 
	\end{theo}
	\begin{proof}
		Firstly, Lemma \ref{Le_iter_bounded_l2} illustrates that when $\sup_{k \geq 0} \norm{D_k}\ff \leq \tilde{M}$, $\beta \geq \max\{ 16M_1, 60\tilde{M} \}$, and $\etak \leq \frac{1}{2\beta}$ for any $k \geq 0$ in Algorithm \ref{Alg:subgradient_sto}, it holds that $\{\Xk\}$ is restricted in $\Omega_{1/6}$. Then every cluster point of $\{\Xk\}$ lies in $\Omega_{1/6}$.

		Now we check the validity of Assumption A and Assumption B in \cite{davis2020stochastic}. The fact that $\{\Xk\}$ is restricted in $\Omega_{1/6}$ directly shows that Assumption A(1) and A(2) hold. The Assumption A(3) and A(4) directly follow from Condition \ref{Assumption_Subgradient_sto}(2) and \ref{Assumption_Subgradient_sto}(3), respectively. In addition, Assumption A.5 follows quickly from the fact that $\ca{D}_f$ is outer-semicontinuous as mentioned in \cite[Lemma 4.1]{davis2020stochastic}.  Furthermore, Assumption B(1) in \cite{davis2020stochastic} is guaranteed by Assumption \ref{Assumption_Subgradient_sto}(4), while Assumption B(2) directly follows from Assumption \ref{Assumption_f}(1) and \cite[Lemma 5.2]{davis2020stochastic}. 
		
		Therefore, from \cite[Theorem 3.2]{davis2020stochastic}, it holds that any cluster point of  $\{\Xk\}$ lies in the set $\{ X \in \Omega_{1/6}: 0 \in \JA(X)[D_f(\A(X))] + \beta X(X\tp X - I_p) \}$ and $h(\Xk)$ converges.
		Similar to the proof in Theorem \ref{The_Equivalence_local}, we obtain that for any $X \in \Omega_{1/6}$, $0 \in \JA(X)[D_f(\A(X))] + \beta X(X\tp X - I_p)$ implies that $X \in \ca{S}_{n,p}$, whenever $\beta \geq 16\hat{M}$. Therefore, we can conclude that  any cluster point of  $\{\Xk\}$ lies in the set $\{ X \in \ca{S}_{n,p}: 0 \in \JA(X)[\D_f(X)]  \}$, and  thus we complete the proof.
	\end{proof}

	\begin{coro}
		\label{Coro_subgradient_convergence_base}
		Suppose Assumption \ref{Assumption_f} holds and Algorithm \ref{Alg:subgradient_sto} satisfies Condition \ref{Assumption_Subgradient_sto}, and $\D_f$ in Algorithm \ref{Assumption_Subgradient_sto} is chosen as $\partial f$. Then $\{h(\Xk)\}$ converges and every cluster point $X^*$ of $\{\Xk\}$ is a first-order stationary point of \ref{Prob_Ori}. 
	\end{coro}
	Corollary \ref{Coro_subgradient_convergence_base} directly follows from Definition \ref{Defin_FOSP} and Theorem \ref{Theo_subgradient_convergence_base}, hence we omit its proof for simplicity.

	\subsection{Stochastic proximal subgradient method}
	
	In this subsection, we consider a special cases of \ref{Prob_Ori}, where the objective function $f$ takes a composite form:
	\begin{equation}
		f(X) = \phi(X) + r(X).
	\end{equation}
	Here,  $\phi$ and $r$ satisfy the  following assumptions.
	\begin{assumpt}
		\label{Assumption_f_composite}
		\begin{enumerate}
			\item $\phi$ is locally Lipschitz continuous and possibly nonconvex. 
			\item $r$ is convex and  $M_r$-Lipschitz continuous over $\bb{R}^{n\times p}$.
			\item The proximal mapping of $r$, defined as 
			\begin{equation}
				\label{Eq_proximal_mapping}
				\prox_{r}(X) := \mathop{\arg\min}\limits_{Y \in \bb{R}^{n\times p}} ~ r(Y) + \frac{1}{2} \norm{Y-X}\fs,
			\end{equation}
			is easy to compute. 
		\end{enumerate}
	\end{assumpt}
	
	Next we present a stochastic proximal subgradient methods for solving \ref{Prob_Ori} in Algorithm \ref{Alg:subgradient_proximal}.
	
	\begin{algorithm}[htbp]
		\begin{algorithmic}[1]   
			\Require Function $f = \phi + r$ and the proximal mapping for $r$;
			\State Choose an initial guess $X_0 \in \Omega_{1/6}$, set $k=0$;
			\While{not terminated}
			\State Compute $D_k $ as an approximated evaluation for $ \partial \phi(\Xk)$;
			\State $\Xkp = \prox_{\eta_k r}\left( \A(\Xk) - \eta_kD_k \right) $;
			\State $k = k+1$;
			\EndWhile
			\State Return $X_k$.
		\end{algorithmic}  
		\caption{A stochastic proximal subgradient methods for solving \ref{Prob_Ori}.}  
		\label{Alg:subgradient_proximal}
	\end{algorithm}
	
	In order to establish the convergence of Algorithm \ref{Alg:subgradient_proximal} through \ref{NEPen}, we need to restrict our attention to the case that satisfies a group of conditions which is very similar to those stated in Condition \ref{Assumption_Subgradient_sto} but under the stochastic environment.
		
	\begin{cond}
		
		\label{Assumption_Subgradient_proximal} 
		In Algorithm \ref{Alg:subgradient_proximal}, we assume 
		\begin{enumerate}
			\item $D_k = G_k + \xi_k$, such that
			\begin{enumerate}
				\item For any subsequence $\{X_{k_j}\}$ that converges to $\tilde{X}$, it holds that 
				\begin{equation*}
					\lim\limits_{j\to \infty} \mathrm{dist}\left( \frac{1}{N}\sum_{j = 1}^N G_{k_j}, \partial \phi(\tilde{X}) \right) = 0. 
				\end{equation*}
				\item $\{\xi_k\}$ is a martingale difference sequence with respect to the $\sigma$-fields $\{\ca{F}_k:= \sigma(X_{j}, \xi_j: j\leq k)\}$. That is, $\{\xi_k\}$ is a sequence of random variables, and there exists $M_{\xi}>0$ such that for any $k >0$, it holds that 
				\begin{equation}
					\bb{E}[\xi_k|\ca{F}_{k-1}] = 0, \quad \text{and}\quad  \bb{E}[\norm{\xi_k}\fs|\ca{F}_{k-1} ] \leq M_{\xi},
				\end{equation}
				almost surely. 
				\item There exists a constant $\tilde{M}$ such that $\sup_{k \geq 0} \norm{D_k}\ff \leq \tilde{M}$. 
			\end{enumerate}
			
			\item  The sequence $\{\eta_k\}$ satisfies 
			\begin{equation*}
				\eta_k > 0, \quad \sum_{k = 0}^{\infty} \eta_k = \infty, \quad \sum_{k = 0}^{\infty} \eta_k^2 < \infty, \quad \sup_{k \geq 0} ~\etak \leq \frac{1}{19(\tilde{M} + M_r)}. 
			\end{equation*} 
		\end{enumerate}
	\end{cond}

	Existing Riemannian proximal-type methods are developed under the additional assumption that $\phi$ is locally Lipschitz continuous over $\ca{S}_{n,p}$. Then in each iteration at $\Xk$, existing Riemannian proximal-type methods have to compute a minimizer of the following Riemannian proximal mapping subproblem, 
	\begin{equation}
		\label{Eq_Rie_proximal_subproblem}
		\min_{Y \in \ca{T}_{X_k}}~ \inner{\grad \phi(\Xk), Y} + r(Y+\Xk) + \frac{1}{2\eta_k} \norm{Y}\fs.
	\end{equation}
	However, the subproblem \eqref{Eq_Rie_proximal_subproblem} is usually not easy to compute, even if $\prox_{\eta_k r}$ has a closed-form computation. Therefore, as illustrated in various of existing works \cite{chen2018proximal,huang2019riemannian,xiao2021penalty}, solving \eqref{Eq_Rie_proximal_subproblem} is usually expensive for these Riemannian proximal-type methods. 
	
	On the other hand, Algorithm \ref{Alg:subgradient_proximal}  only requires to compute the $\prox_{\eta_k r}\left( \A(\Xk) - \eta_kD_k \right)$  rather than the Riemannian proximal mapping subproblem \eqref{Eq_Rie_proximal_subproblem} in each iteration. As a result, the updating formula in Algorithm \ref{Alg:subgradient_proximal} only involves  matrix-matrix multiplications, thus avoiding the bottleneck in solving the Riemannian proximal subproblem  \eqref{Eq_Rie_proximal_subproblem}.
	
	In the rest of this subsection, we prove the global convergence of Algorithm \ref{Alg:subgradient_proximal}.
	The proof for the convergence properties is based on the differential inclusion techniques in \cite{bolte2021conservative,davis2020stochastic}, where we employ the constraint dissolving function \ref{NEPen} as the merit function. 
	Let $\Yk := \A(\Xk)$, the optimality condition of \eqref{Eq_proximal_mapping} illustrates that there exists $T_k \in \partial r(\Xkp)$ such that $\Xkp = \A(\Xk) - \eta_k(D_k + T_k)$. As shown later in Proposition  \ref{Prop_prox_subgradient_AX_convergence}, the sequence $\{\Yk\}$ follows the following updating scheme, 
	\begin{equation}
		\label{Eq_proximal_subgradient_iter}
		\begin{aligned}
			\Ykp - \Yk 
			={}& -\eta_k \left(\JA(\Yk)[G_k + T_k] + \beta \Yk(\Yk\tp\Yk - I_p)\right) \\
			& -  \underbrace{\Big(- \eta_k\beta \Yk(\Yk\tp\Yk - I_p)\Big)}_{=:E_{1,k}} - \underbrace{\Big(\Ykp - \Yk - \eta_k \JA(\Yk)[D_k + T_k]\Big)}_{=:E_{2,k}}\\
			& - \underbrace{\Big(-\eta_k\JA(\Yk)[\xi_k]\Big)}_{=:E_{3,k}}\\
			={}& -\eta_k \left(W_k + E_k\right), 
		\end{aligned}
	\end{equation}
	where $W_k =\JA(\Yk)[G_k + T_k] + \beta \Yk(\Yk\tp\Yk - I_p) $ and $E_k = \frac{E_{1, k} + E_{2, k} + E_{3, k}}{\eta_k}$. 
	
	When  $W_k$ is an approximated evaluation for $\partial h(\Yk)$ and $E_k$ is regarded as an error term,  $\{\Yk\}$ can be regarded as a discretization for the trajectory $\gamma: \bb{R}_+ \to \bb{R}^{n\times p}$ that satisfies the following differential inclusion,
	\begin{equation}
		\label{Eq_proximal_subgradient_differential_inclusion}
		\dot{\gamma}(t) \in  -\partial h(\gamma(t)). 
	\end{equation}
	Then  by  using similar techniques as  in Theorem \ref{Theo_subgradient_convergence_base}, we show that $\{\Yk\}$ tracks the trajectory of the differential inclusion \eqref{Eq_proximal_subgradient_differential_inclusion} to  prove the global convergence of Algorithm \ref{Alg:subgradient_proximal}.

	\begin{lem}
		\label{Le_prox_subgradient_iter_sum_bounded}
		Suppose Assumption \ref{Assumption_f_composite} and Condition \ref{Assumption_Subgradient_proximal} hold in Algorithm \ref{Alg:subgradient_proximal}. Then almost surely,  any $\{\Xk\}$ generated by Algorithm \ref{Alg:subgradient_proximal} lies in $\Omega_{1/6}$.
	\end{lem}
	\begin{proof}
		For any $\{\Xk\}$ generated by  Algorithm \ref{Alg:subgradient_proximal}, step 4 in Algorithm \ref{Alg:subgradient_proximal} and the optimality condition of \eqref{Eq_proximal_mapping} implies that $ 0 \in \Xkp - \A(\Xk) + \eta_k D_k +  \eta_k\partial r(\Xkp)$. 
		Therefore, for any $k >0$, there exists $T_k \in \partial r(\Xkp)$ such that $\Xkp - \A(\Xk) = -\eta_k (D_k + T_k)$. As a result, we obtain
		\begin{equation}
			\label{Eq_Le_prox_subgradient_iter_sum_bounded_0}
			\begin{aligned}
				&\norm{\Xkp\tp \Xkp - I_p}\ff \\
				\leq{}&  \norm{\A(\Xk)\tp \A(\Xk)  - I_p - 2\eta_k \Phi((D_k + T_k)\tp\A(\Xk)) + \eta_k^2 (D_k + T_k)\tp(D_k + T_k) }\ff \\
				\leq{}& \norm{\A(\Xk)\tp \A(\Xk)  - I_p }\ff + 2\eta_k \norm{D_k+T_k}\ff \norm{\A(\Xk)}_2 + \eta_k^2 \norm{D_k+ T_k}\fs \\
				\leq{}& \norm{\Xk\tp \Xk - I_p}\ff^3 + 3\eta_k (\tilde{M} + M_r) 
				\leq \frac{1}{36}\norm{\Xk\tp \Xk - I_p}\ff + 3\eta_k (\tilde{M} + M_r). 
			\end{aligned}
		\end{equation}
		When $\norm{\Xk\tp \Xk - I_p}\ff \leq \frac{1}{6}$, it holds that 
		\begin{equation*}
			\norm{\Xkp\tp \Xkp - I_p}\ff \leq \frac{1}{216} + 3\eta_k (\tilde{M} + M_r) \leq  \frac{1}{216} + \frac{3}{19} < \frac{1}{6},
		\end{equation*}
		hence we  complete the first part of the proof by induction. 
		
		Moreover, it directly follows from \eqref{Eq_Le_prox_subgradient_iter_sum_bounded_0}  that
		\begin{equation*}
			\norm{\Xkp\tp \Xkp - I_p}\fs \leq \left(\frac{1}{36}\norm{\Xk\tp \Xk - I_p}\ff + 3\eta_k (\tilde{M} + M_r)\right)^2.
		\end{equation*}
		Together with Cauchy's inequality, we have that  
		\begin{equation}
			\label{Eq_Le_prox_subgradient_iter_sum_bounded_1}
			\norm{\Xkp\tp \Xkp - I_p}\fs - \frac{1}{648} \norm{\Xk\tp \Xk - I_p}\fs \leq  18 \eta_k^2 (\tilde{M} + M_r)^2. 
		\end{equation}
		Then by \eqref{Eq_Le_prox_subgradient_iter_sum_bounded_1}, it holds that 
		\begin{equation}
			\label{Eq_Le_prox_subgradient_iter_sum_bounded_2}
			\sum_{k = 0}^{\infty}\norm{\Xk\tp \Xk - I_p }\fs \leq 19  (\tilde{M} + M_r)^2 \sum_{k = 0}^{\infty} \eta_k^2 < \frac{1}{6}, 
		\end{equation}
		and we complete the proof. 
	\end{proof}

	\begin{prop}
		\label{Prop_prox_subgradient_AX_convergence}
		Suppose Assumption \ref{Assumption_f}, Assumption \ref{Assumption_f_composite} and Condition \ref{Assumption_Subgradient_proximal} holds in Algorithm \ref{Alg:subgradient_proximal}. Then for any $\{\Xk\}$ generated by Algorithm \ref{Alg:subgradient_proximal}, any cluster point of $\{\A(\Xk)\}$ is a first-order stationary point of \ref{Prob_Ori}, and $\{h(\A(\Xk))\}$ converges. 
	\end{prop}
	\begin{proof}
		Let $\Yk = \A(\Xk)$, it holds from the smoothness of $\A$ that 
		\begin{equation}
			\Ykp - \Yk = \A(\Xkp) - \A(\Xk)  = \A(\Yk - \eta_k(G_k+ \xi_k + T_k) ) - \Yk.
		\end{equation}
		As a result, from the definition of $\{Y_k\}$, we obtain that 
		\begin{equation}
			\label{Eq_Prop_prox_subgradient_AX_convergence_1}
			\begin{aligned}
				\Ykp - \Yk ={}& -\eta_k \left(\JA(\Yk)[G_k + T_k] + \beta \Yk(\Yk\tp\Yk - I_p)\right) \\
				& -  \underbrace{\Big(- \eta_k\beta \Yk(\Yk\tp\Yk - I_p)\Big)}_{ =:E_{1,k}} - \underbrace{\Big(\Ykp - \Yk - \eta_k \JA(\Yk)[D_k + T_k]\Big)}_{ =:E_{2,k}}\\
				& - \underbrace{\Big(-\eta_k\JA(\Yk)[\xi_k]\Big)}_{ =:E_{3,k}}. 
			\end{aligned}
		\end{equation}
		We first prove that $\sum_{k = 0}^{\infty} E_{1, k}$ exists and is finite. 
		From the expression of $E_{1, k}$, it holds that 
		\begin{equation*}
			\begin{aligned}
				&\sum_{k = 0}^{\infty} \norm{E_{1, k}}\ff \leq \sum_{k = 0}^{\infty} \beta\eta_k \norm{\Yk}_2 \norm{\Yk\tp \Yk - I_p}\ff \overset{(i)}{\leq} \frac{7\beta }{6} \sum_{k = 0}^{\infty}  \eta_k \norm{\Xk\tp \Xk - I_p}\ff^3  \\
				\leq{}& \frac{7\beta\sup_{k \geq 0}\eta_k  }{36}  \sum_{k = 0}^{\infty}   \norm{\Xk\tp \Xk - I_p}\ff^2 \overset{(ii)}{<} \infty. 
			\end{aligned}
		\end{equation*}
		Here $(i)$ follows from Proposition \ref{Prop_ATA} and $(ii)$ is implied by \eqref{Eq_Le_prox_subgradient_iter_sum_bounded_2} in Lemma \ref{Le_prox_subgradient_iter_sum_bounded}. Then from the controlled convergence theorem, we obtain that $\sum_{k = 0}^{\infty} E_{1, k}$ exists and takes a finite value. 
		
	Next prove that $\sum_{k = 0}^{\infty} E_{2, k}$ exists and takes a finite value. Since $\A$ is smooth over $\bb{R}^{n\times p}$, let $M_A = \sup_{Y, Z \in \Omega_1, Y\neq Z  } 
	{\norm{\JA(Y) - \JA(Z)}\ff}/{\norm{Y-Z}\ff}$. Then  for any $X \in \Omega_{1/6}$ and any $\norm{D}\ff \leq 1$,   it holds that $\norm{\A(X+D) - \A(X) - \JA(X)[D]}\ff \leq M_A \norm{D}\fs$. 
		From the definition of $\{Y_k\}$ and $M_A$, we get
		\begin{equation}
			\label{Eq_Prop_prox_subgradient_AX_convergence_0}
			\begin{aligned}
				&\norm{\Ykp - \Yk - \eta_k \JA(\Yk)[G_k+ \xi_k + T_k]}\ff \\
				\leq{}& \norm{\A(\Yk - \eta_k(G_k+ \xi_k + T_k) ) - \Yk - \eta_k \JA(\Yk)[G_k+ \xi_k + T_k]}\ff\\
				\leq{}& \norm{\A(\Yk - \eta_k(G_k+ \xi_k + T_k) ) - \A(\Yk) - \eta_k \JA(\Yk)[G_k+ \xi_k + T_k]}\ff + \norm{\A(\Yk) - \Yk}\ff \\
				\leq{}& (\tilde{M} + M_r)^2 M_A\eta_k^2  + \norm{\Yk\tp \Yk - I_p}\ff 
				\leq (\tilde{M} + M_r)^2M_A\eta_k^2 +  \norm{\Xk\tp \Xk - I_p}\ff^3. 
			\end{aligned}
		\end{equation}
		Therefore,  from the expression of $E_{2, k}$, it holds that 
		\begin{equation*}
			\begin{aligned}
				&\sum_{k = 0}^{\infty} \norm{E_{2, k}}\ff
				\overset{(iii)}{\leq} \sum_{k = 0}^{\infty} \left((\tilde{M} + M_r)^2 M_A\eta_k^2  + \norm{\Xk\tp \Xk - I_p}\ff^3\right)\\
				\leq{}& (\tilde{M} + M_r)^2M_A\sum_{k = 0}^{\infty} \eta_k^2 +  \frac{1}{6} \sum_{k = 0}^{\infty}\norm{\Xk\tp \Xk - I_p}\fs\\
				\overset{(iv)}{\leq}{}&  5(\tilde{M} + M_r)^2M_A\sum_{k = 0}^{\infty} \eta_k^2 + \frac{19(\tilde{M} + M_r)^2}{6}\sum_{k = 0}^{\infty} \eta_k^2< \infty. 
			\end{aligned}
		\end{equation*}
		Here $(iii)$ follows from \eqref{Eq_Prop_prox_subgradient_AX_convergence_0}, and $(iv)$ holds directly from \eqref{Eq_Le_prox_subgradient_iter_sum_bounded_2}. Then by using the controlled convergence theorem again, we obtain that $\sum_{k = 0}^{\infty} E_{2, k}$ exists and takes a finite value. 
		
		Finally,  from the definition of ${E_{3, k}}$, $\{\frac{1}{\eta_k}E_{3, k}\} = \{ -\JA(Y_k)[\xi_k] \}$ is a martingale difference sequence. Therefore, it follows from \cite[Theorem 5.3.33]{dembo2010probability} that $\sum_{k = 0}^{\infty} E_{3, k}$ exists and takes finite values. As a result, for the error sequence $\{E_{k}\}$ defined by $\{E_k\} = \{ \frac{1}{\eta_k}(E_{1, k} + E_{2, k} + E_{3, k}) \}$ in \eqref{Eq_proximal_subgradient_iter}, it holds that $\sum_{k = 0}^{\infty} \eta_k E_k $ exists and takes a finite value. 
		
		Now we check the validity of Assumption A and Assumption B in \cite{davis2020stochastic}. Lemma \ref{Le_prox_subgradient_iter_sum_bounded} states that $\{\Xk\}$ is restricted in $\Omega_{1/6}$. Then  Proposition \ref{Prop_ATA} shows that $\{\Yk\} \subset \Omega_{1/6}$, and hence  Assumption A(1) and A(2) hold. Assumption  A(3) holds from the fact that $\sum_{k = 0}^{\infty} \eta_k E_k $ exists and takes a finite value. 
		Assumption A(4) directly follows from Condition  \ref{Assumption_Subgradient_proximal}(3). In addition, Assumption A.5 follows quickly from the fact that $\partial h$ is outer-semicontinuous as mentioned in \cite[Lemma 4.1]{davis2020stochastic}. 
		
		Furthermore, Assumption B(1) in \cite{davis2020stochastic} is guaranteed by Assumption \ref{Assumption_f}(2) and Theorem \ref{The_Equivalence_local}, while Assumption B(2) directly follows from Assumption \ref{Assumption_f}(1) and \cite[Lemma 5.2]{davis2020stochastic}. 
		
		Therefore, from \cite[Theorem 3.2]{davis2020stochastic} it holds that any cluster point of  $\{\Yk\}$ lies in the set $\{ X \in \Omega: 0 \in \partial h(X) \}$ and $h(\Yk)$ converges. 
	\end{proof}

	\begin{theo}
		\label{Theo_prox_subgradient_convergence}
		Suppose Assumption \ref{Assumption_f}, Assumption \ref{Assumption_f_composite} and Condition \ref{Assumption_Subgradient_proximal} holds in Algorithm \ref{Alg:subgradient_proximal}. Let $\{\Xk\}$ be the sequence generated by Algorithm \ref{Alg:subgradient_proximal}. Then any cluster point of $\{\Xk\}$ is a first-order stationary point of \ref{Prob_Ori}, and $\{h(\Xk)\}$ converges. 
	\end{theo}
	\begin{proof}
		Proposition \ref{Prop_prox_subgradient_AX_convergence} states that any cluster point of $\{\Yk\}$ is a first-order stationary point of \ref{Prob_Ori}, and $\{h(\Yk)\}$ converges. 
		Lemma \ref{Le_prox_subgradient_iter_sum_bounded} illustrates that 
		$ \lim_{k \to \infty} \|\Xk\tp \Xk - I_p\|\ff = 0$. Then together  with Lemma \ref{Le_prox_subgradient_iter_sum_bounded}, we can conclude that $\lim_{k\to \infty} \norm{\Yk - \Xk}\ff = 0$. As a result, any cluster point of  $\{\Xk\}$ is a first-order stationary point of \ref{Prob_Ori} and $h(\Xk)$ converges. 
	\end{proof}

	\section{Numerical Experiments}

	In this section, we apply \ref{Prob_Ori} to train an orthogonally constrained convolutional neural network (OCNN) and compare its numerical performance with existing practical approaches.  All the numerical experiments in this section are run on a server with Intel Xeon 6342 CPU and NVIDIA GeForce RTX 3090 GPU running PyTorch 1.8.0.

	In our numerical experiments, to illustrate the practical performance of \ref{NEPen}, we choose the well-known CIFAR image classification dataset \cite{krizhevsky2009learning}. We develop the VGG-19 network \cite{simonyan2014very} based on the PyTorch examples on image classification, where the weights of each layer are restricted on the Stiefel manifold. Let $W_i\in \bb{R}^{n_{k_{1}, i}\times  \cdots \times n_{k_{l}, i}}$ 
	be the weights of the $i$-th layer and $\tilde{W}$ refers to all the other trainable parameters, training the following orthogonal constrained neural network can be expressed as the following optimization problem, 
	\begin{equation}
		\begin{aligned}
			\min_{W_1,...,W_l, \hat{W}}\quad & f(W_1,..., W_l, \hat{W}) = \phi(W_1,..., W_l, \hat{W}) + \gamma r(W_1,..., W_l, \hat{W}),\\
			\text{s. t.} \quad & \mathrm{reshape}(W_i)\tp \mathrm{reshape}(W_i) = I_{n_{k_{l},i}}, \quad i = 1,..., l. 
		\end{aligned}
	\end{equation}
	Here $\mathrm{reshape}(W_i)$ is the operator that reshape the $l$-th order tensor $W_i \in \bb{R}^{n_{k_{1}, i}\times  \cdots \times n_{k_{l}, i}}$ to the matrix in $\bb{R}^{n_{k_1,i}n_{k_2, i}...n_{k_{l-1},i} \times n_{k_l}}$. Moreover, we choose the regularization term $r$ as $\gamma \sum_{i = 1}^l \norm{W_{i}}_1 + \norm{\tilde{W}}_1$ to promote the sparsity of parameters in neural network \cite{zhang2016l1}.

	We compare the numerical performance of Algorithm \ref{Alg:subgradient_sto} and Algorithm \ref{Alg:subgradient_proximal} with some existing efficient adaptive stochastic gradient methods, such as  Ranger \cite{wright2021ranger21} and Yogi \cite{zaheer2018adaptive}. Moreover, we also test the performance of  existing Riemannian stochastic gradient methods for \eqref{Prob_Ori}. We choose the Riemannian stochastic subgradient method (R-sgd) \cite{li2021weakly,bolte2021conservative} and the Riemannian stochastic proximal gradient method (R-proxsgd) proposed by \cite{wang2020riemannian}. In all the Riemannian optimization approaches, the retraction is chosen as the polar decomposition \cite{Absil2009optimization}. 
	
	 In our numerical experiments, we set the batchsize as $128$ in all the test instances and choose the parameter $\gamma$ for the regularization term as $\gamma = 10^{-5}$. Moreover, at $k$-th epoch, we choose the stepsize as $\frac{\eta_0}{0.1 \cdot k+1}$ for all the tested algorithms. Follows the settings in \cite{castera2021inertial}, we use the grid search for a good initial stepsize $\eta_0$:  we select the initial stepsize $\eta_0$ from $\{k_1 \times 10^{-k_2}: k_1 = 1,3,5,7,9, ~k_2 = 1,2  \}$, and choose $\eta_0$ as the one that most increases the accuracy after $20$ epochs for each algorithm, respectively.  In addition, we set the penalty parameter as $0.1$ for all the orthogonally constrained layers in \ref{NEPen}. All the compared algorithms are initialized with the same random weights, which are randomly generated by the default initialization function of PyTorch and are projected to the feasible region by the polar decomposition.

		\begin{figure}[!htbp]
		\centering
		\subfigure[Test accuracy]{
			\begin{minipage}[t]{0.33\linewidth}
				\centering
				\includegraphics[width=\linewidth]{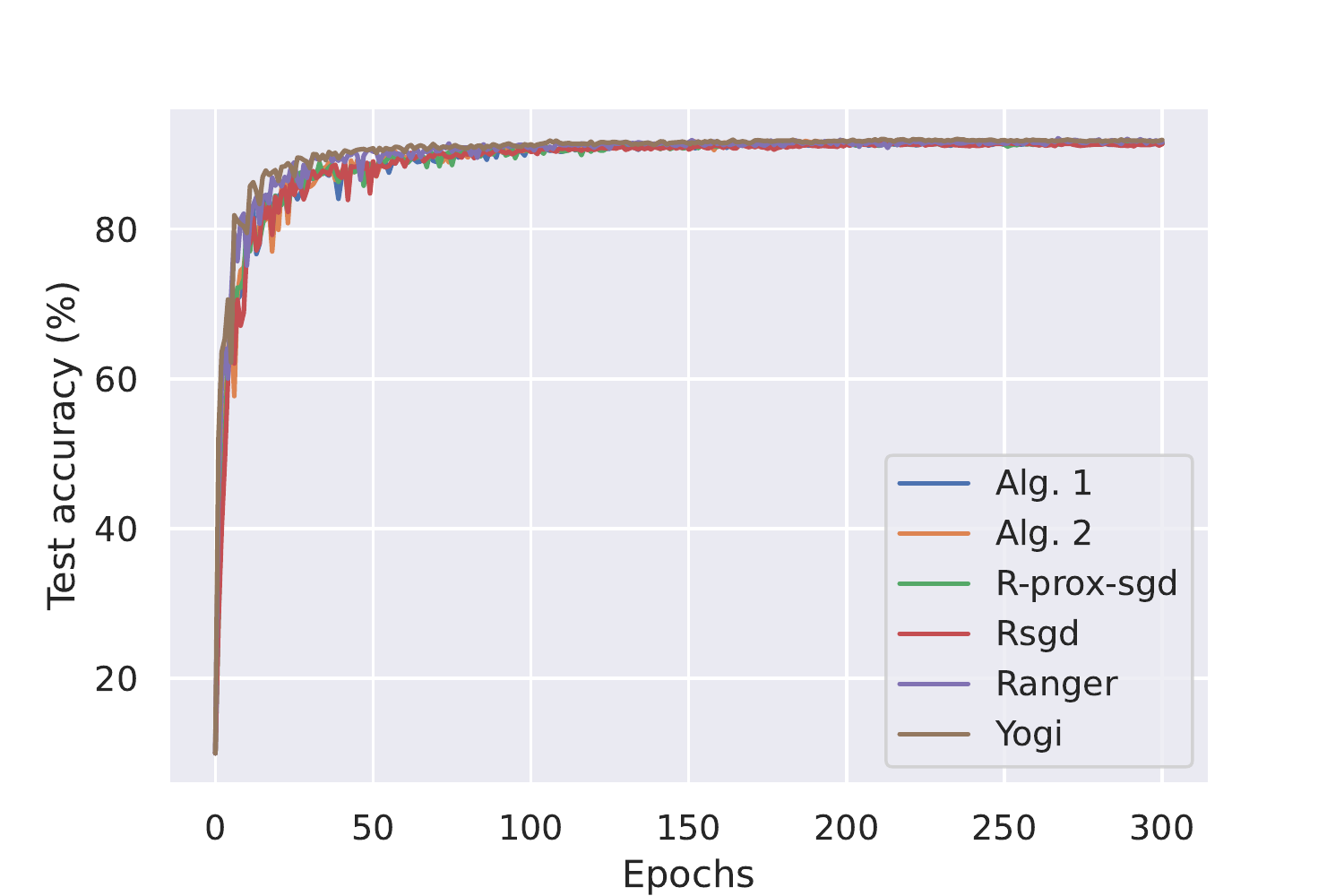}
				\label{Fig:CIFAR10_0_2}
			\end{minipage}%
		}%
		\subfigure[Train loss]{
			\begin{minipage}[t]{0.33\linewidth}
				\centering
				\includegraphics[width=\linewidth]{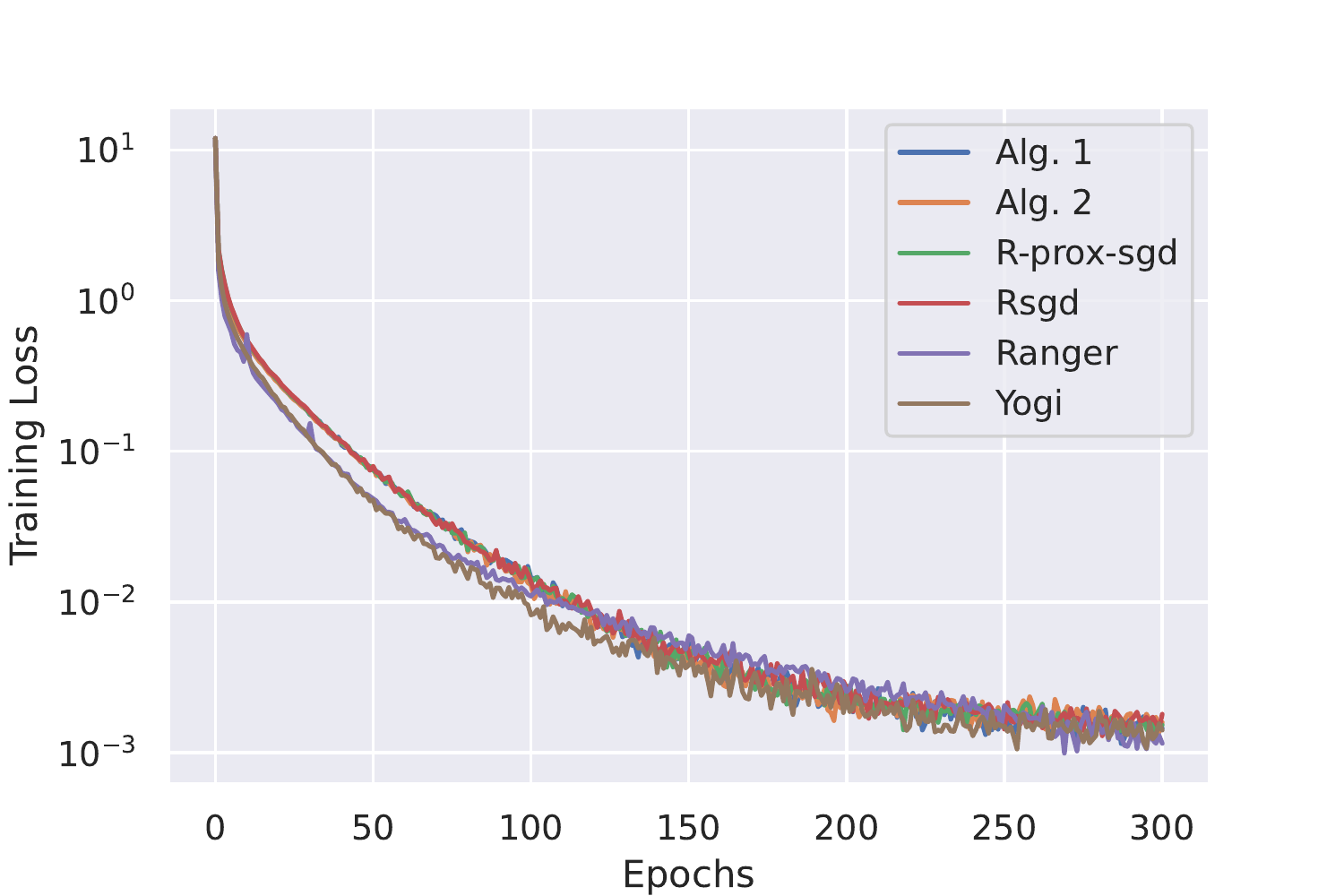}
				\label{Fig:CIFAR10_0_3}
			\end{minipage}%
		}%
		\subfigure[Sparsity]{
			\begin{minipage}[t]{0.33\linewidth}
				\centering
				\includegraphics[width=\linewidth]{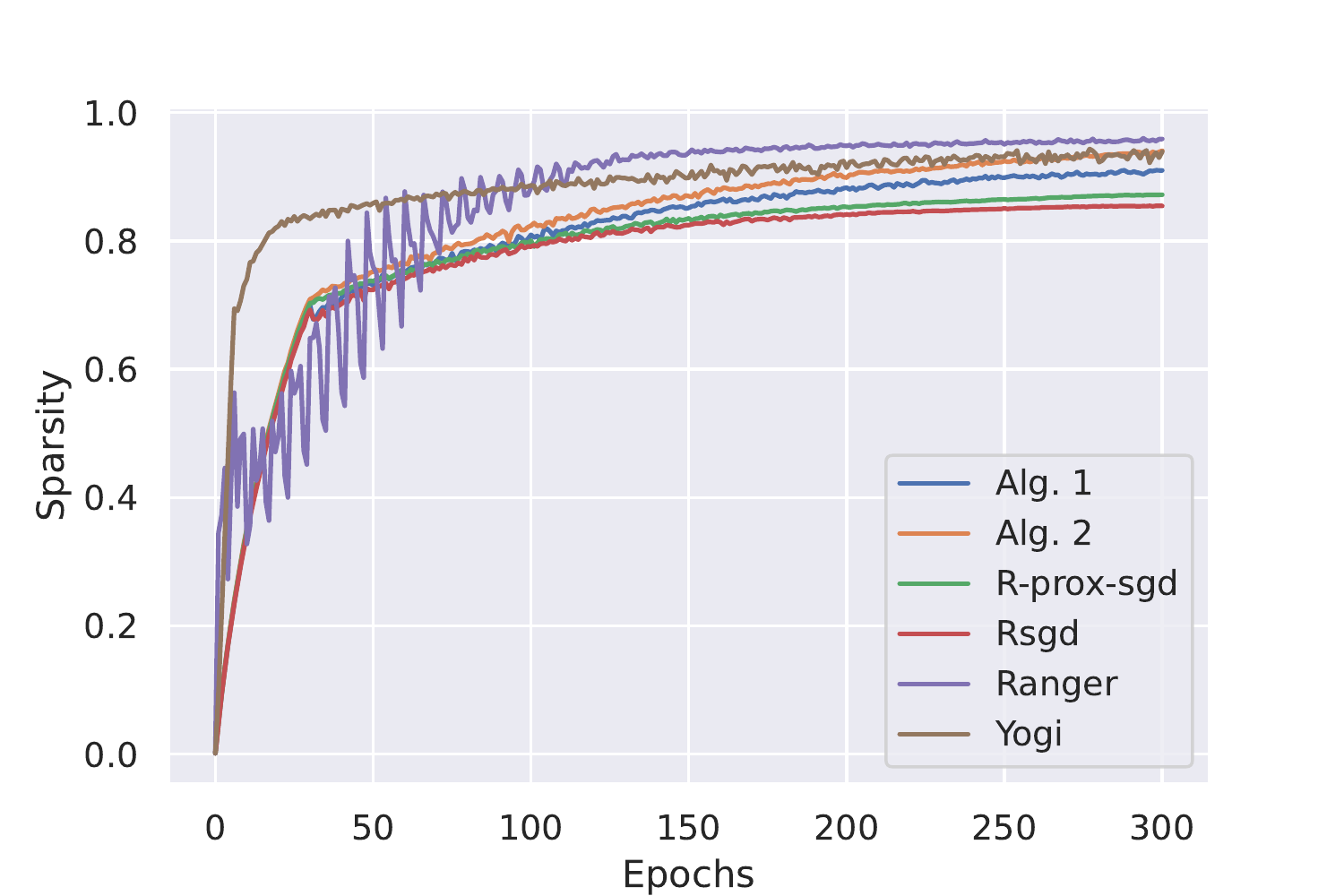}
				\label{Fig:CIFAR10_0_4}
			\end{minipage}%
		}%

		\subfigure[Test accuracy]{
			\begin{minipage}[t]{0.33\linewidth}
				\centering
				\includegraphics[width=\linewidth]{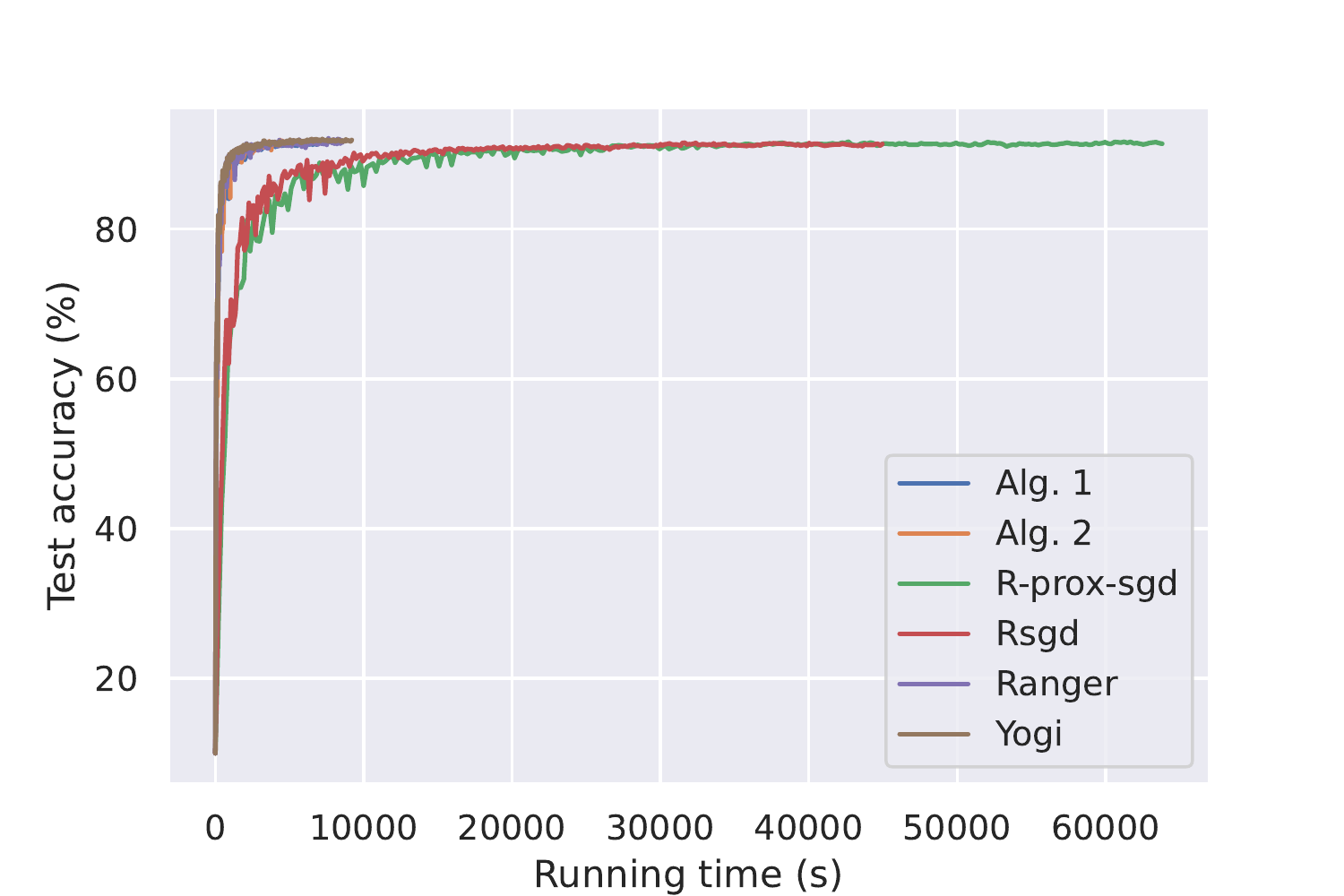}
				\label{Fig:CIFAR10_1_2}
			\end{minipage}%
		}%
		\subfigure[Train loss]{
			\begin{minipage}[t]{0.33\linewidth}
				\centering
				\includegraphics[width=\linewidth]{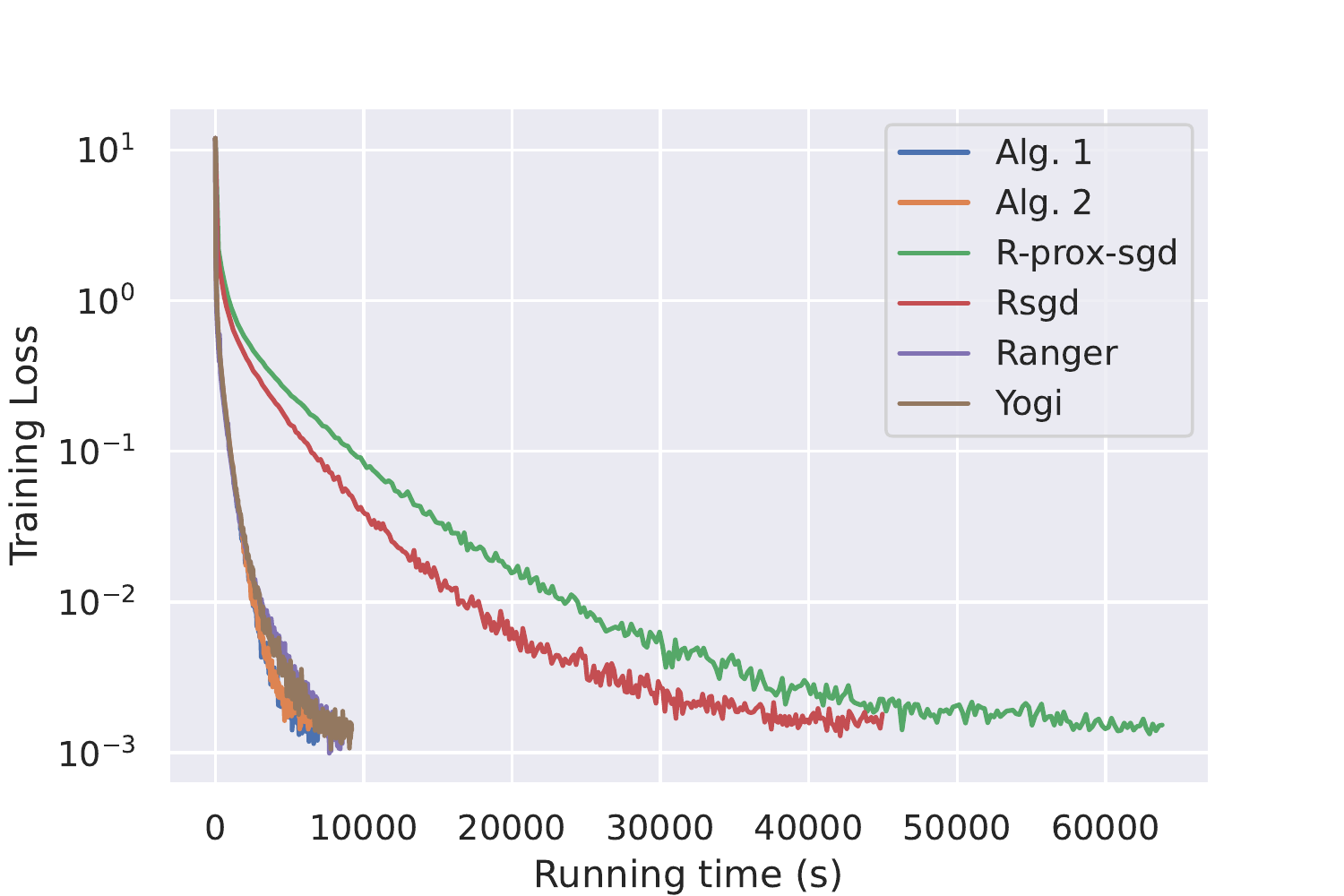}
				\label{Fig:CIFAR10_1_3}
			\end{minipage}%
		}%
		\subfigure[Sparsity]{
			\begin{minipage}[t]{0.33\linewidth}
				\centering
				\includegraphics[width=\linewidth]{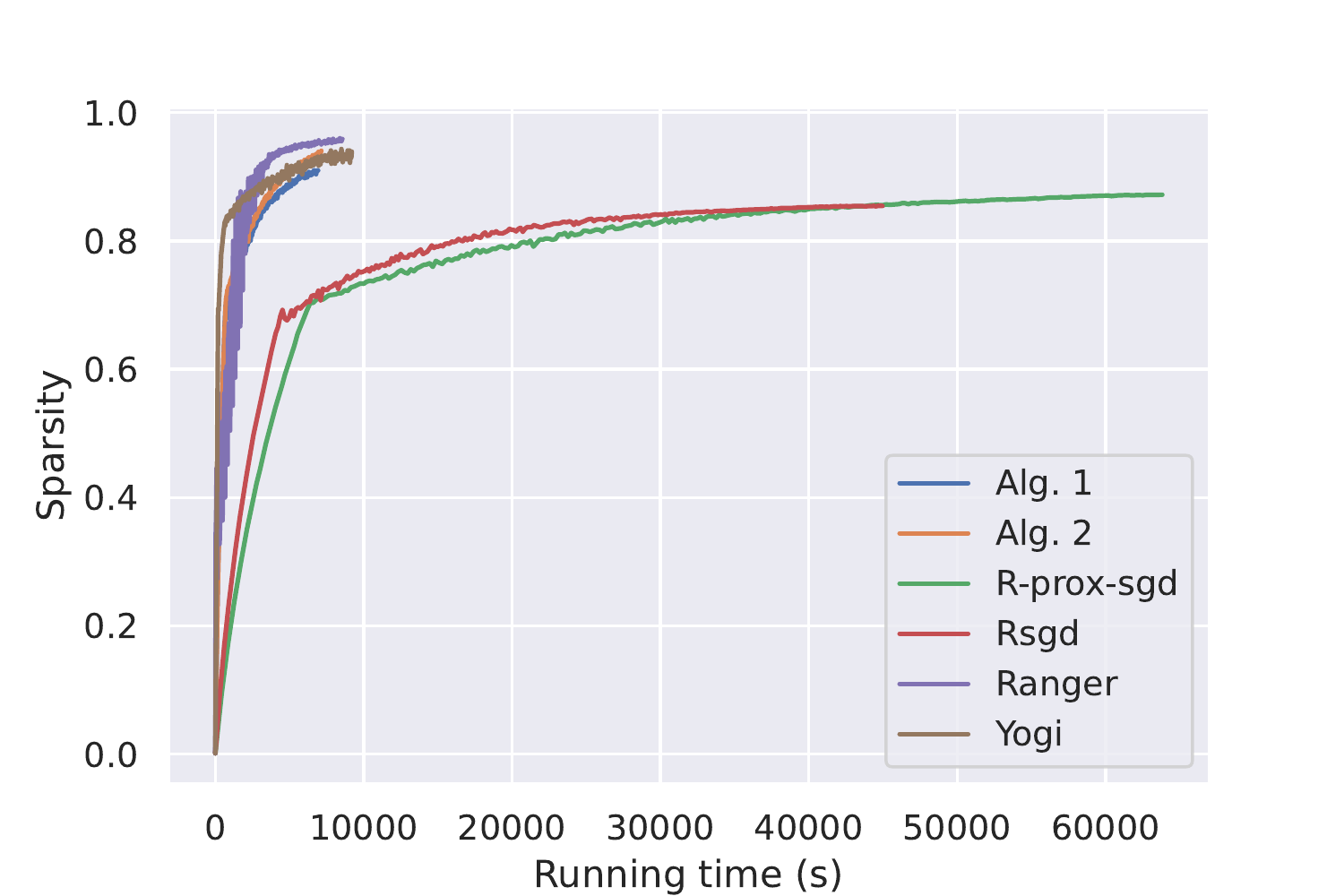}
				\label{Fig:CIFAR10_1_4}
			\end{minipage}%
		}%
		\caption{Comparison of Algorithm \ref{Alg:subgradient_sto} and Algorithm \ref{Alg:subgradient_proximal} with state-of-the-art algorithms R-sgd and R-proxsgd on CIFAR10 dataset.}
		\label{Fig_cifar10}
	\end{figure}

			\begin{figure}[!htbp]
		\centering
		\subfigure[Test accuracy]{
			\begin{minipage}[t]{0.33\linewidth}
				\centering
				\includegraphics[width=\linewidth]{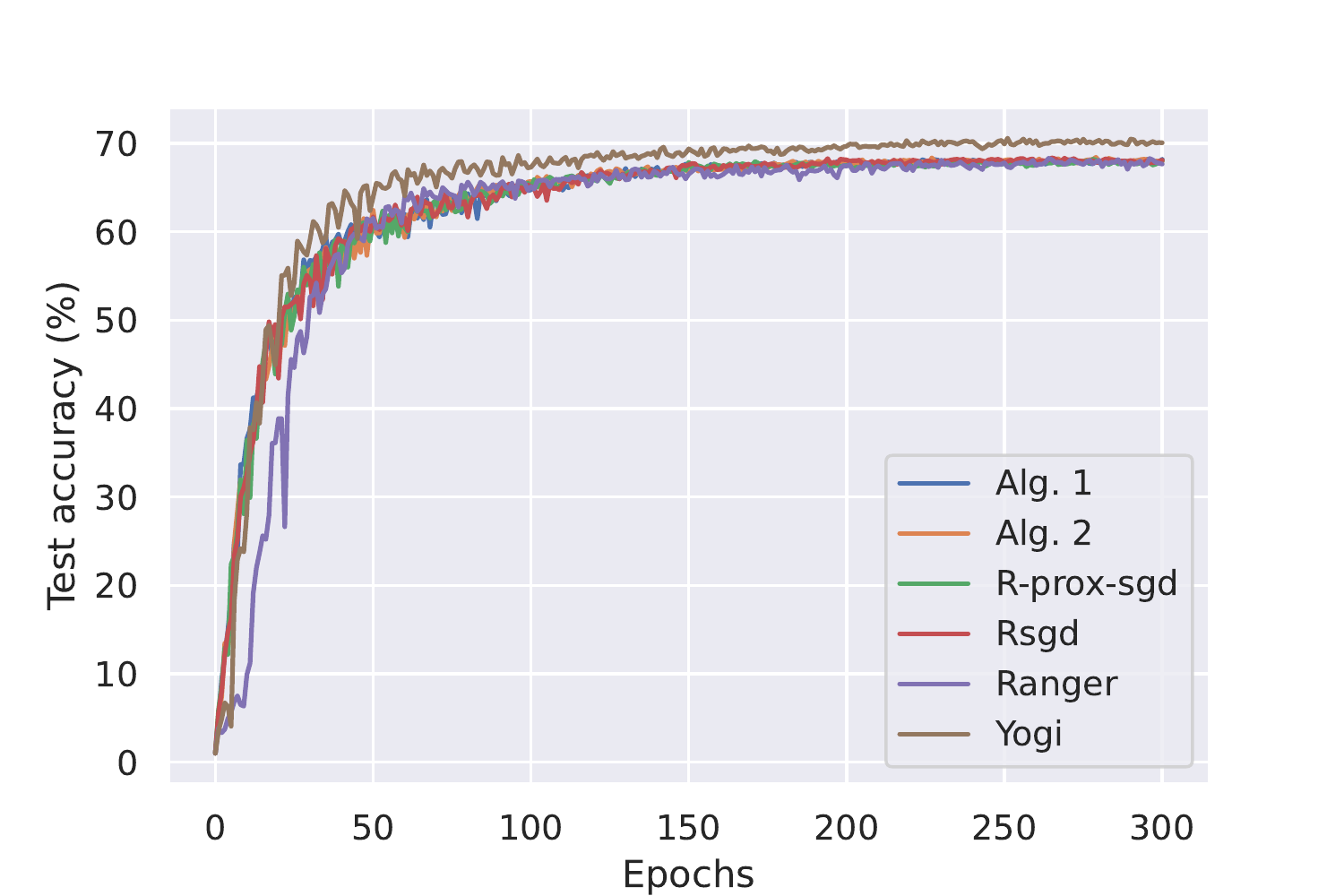}
				\label{Fig:CIFAR100_0_2}
			\end{minipage}%
		}%
		\subfigure[Train loss]{
			\begin{minipage}[t]{0.33\linewidth}
				\centering
				\includegraphics[width=\linewidth]{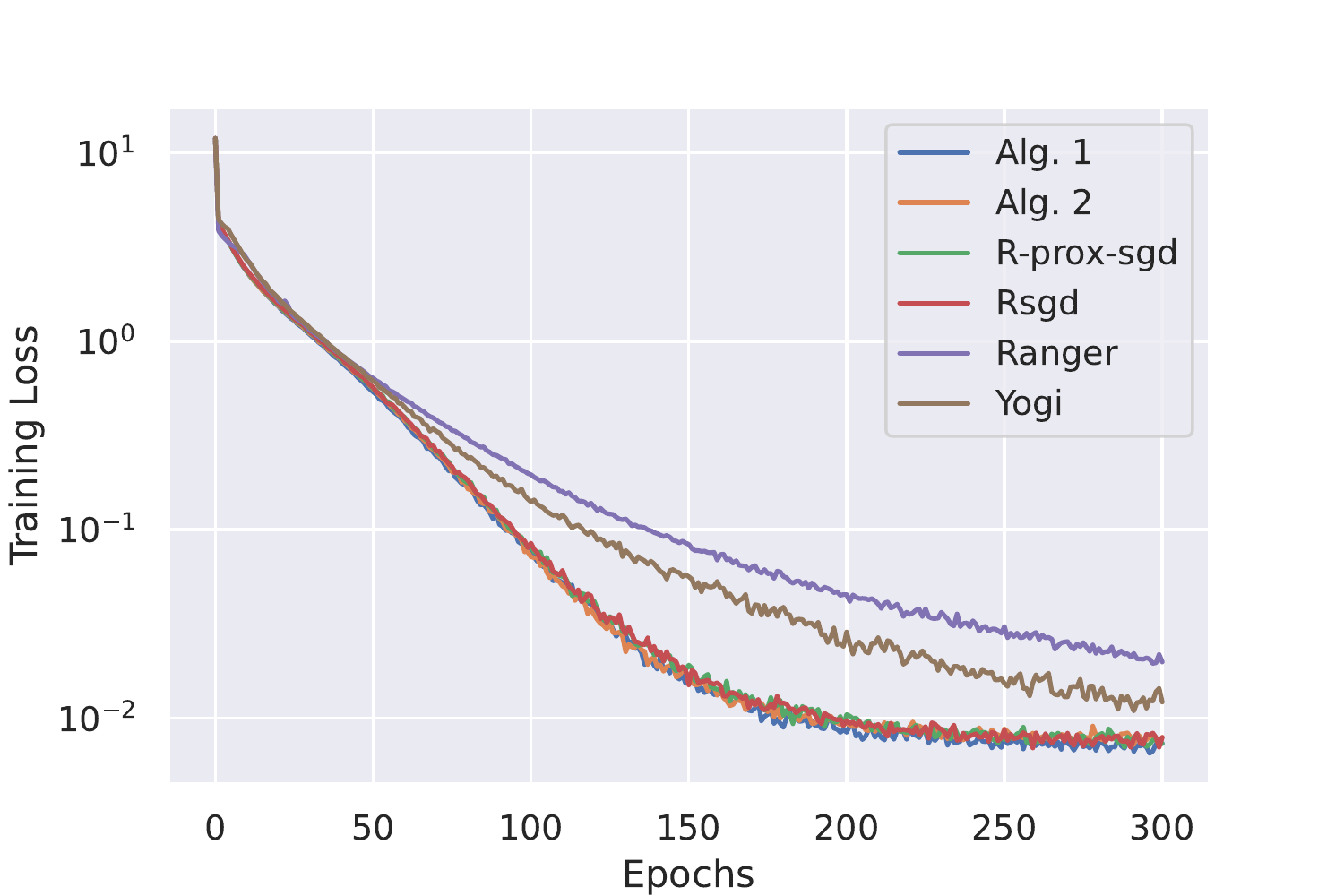}
				\label{Fig:CIFAR100_0_3}
			\end{minipage}%
		}%
		\subfigure[Sparsity]{
			\begin{minipage}[t]{0.33\linewidth}
				\centering
				\includegraphics[width=\linewidth]{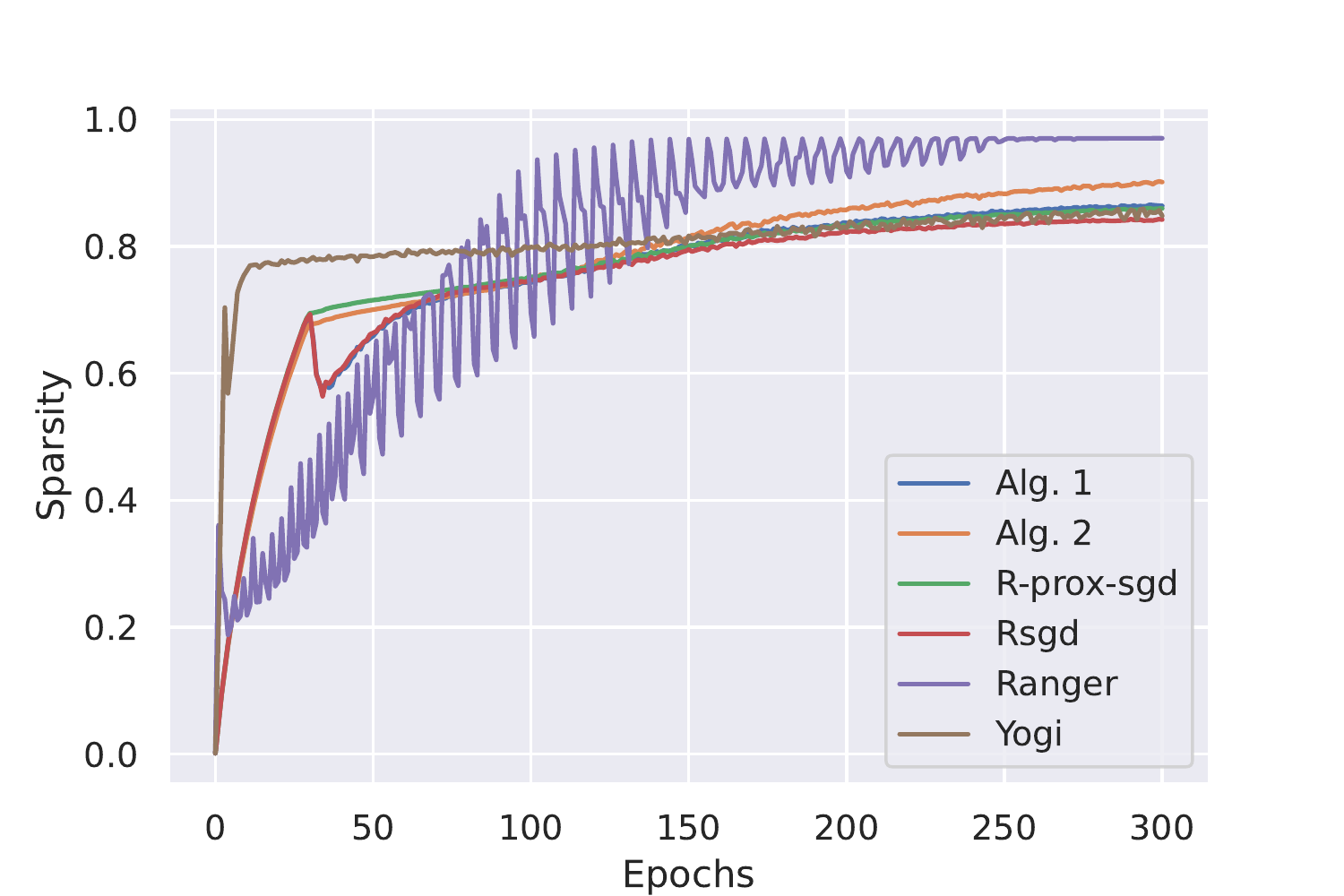}
				\label{Fig:CIFAR100_0_4}
			\end{minipage}%
		}%

		\subfigure[Test accuracy]{
			\begin{minipage}[t]{0.33\linewidth}
				\centering
				\includegraphics[width=\linewidth]{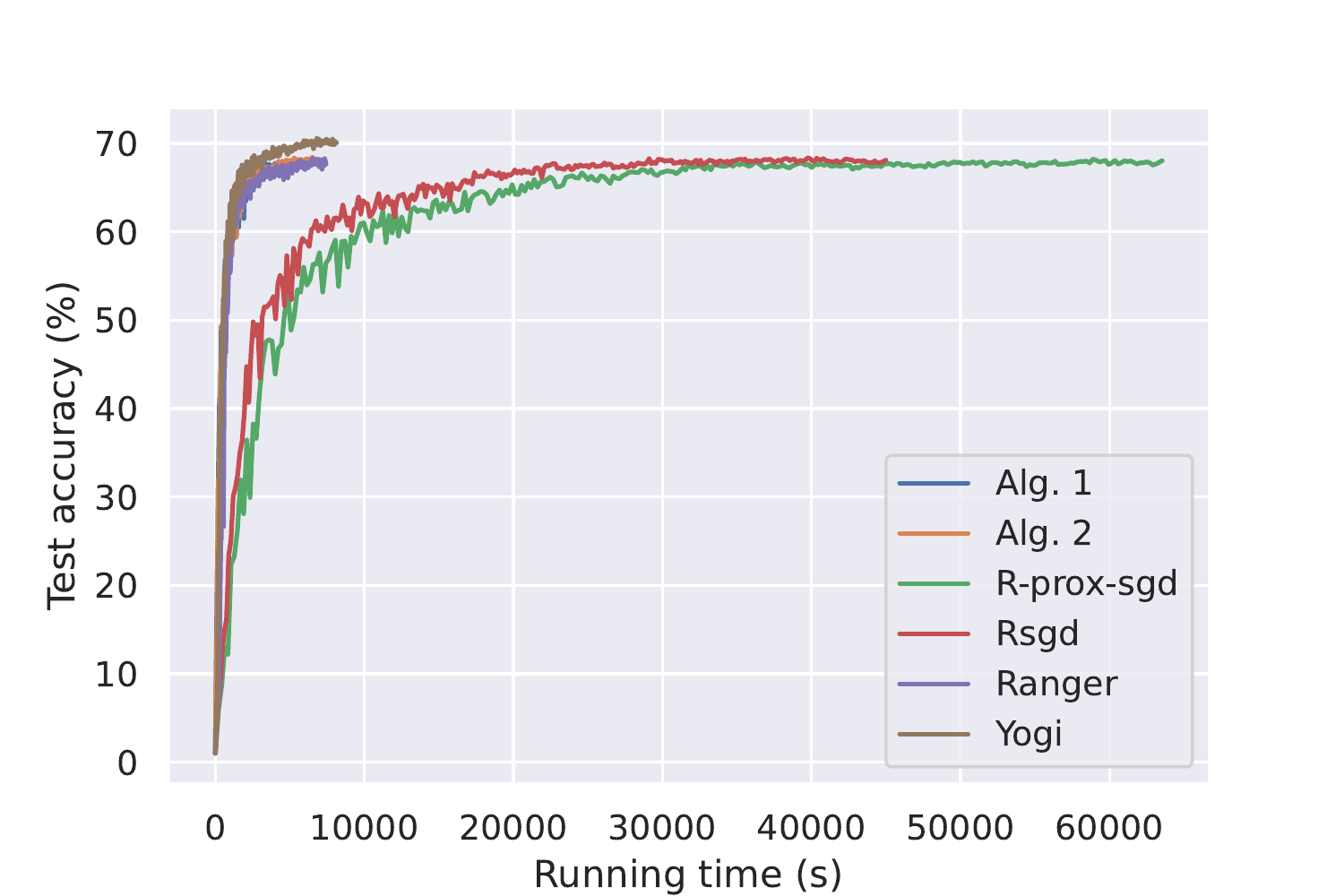}
				\label{Fig:CIFAR100_1_2}
			\end{minipage}%
		}%
		\subfigure[Train loss]{
			\begin{minipage}[t]{0.33\linewidth}
				\centering
				\includegraphics[width=\linewidth]{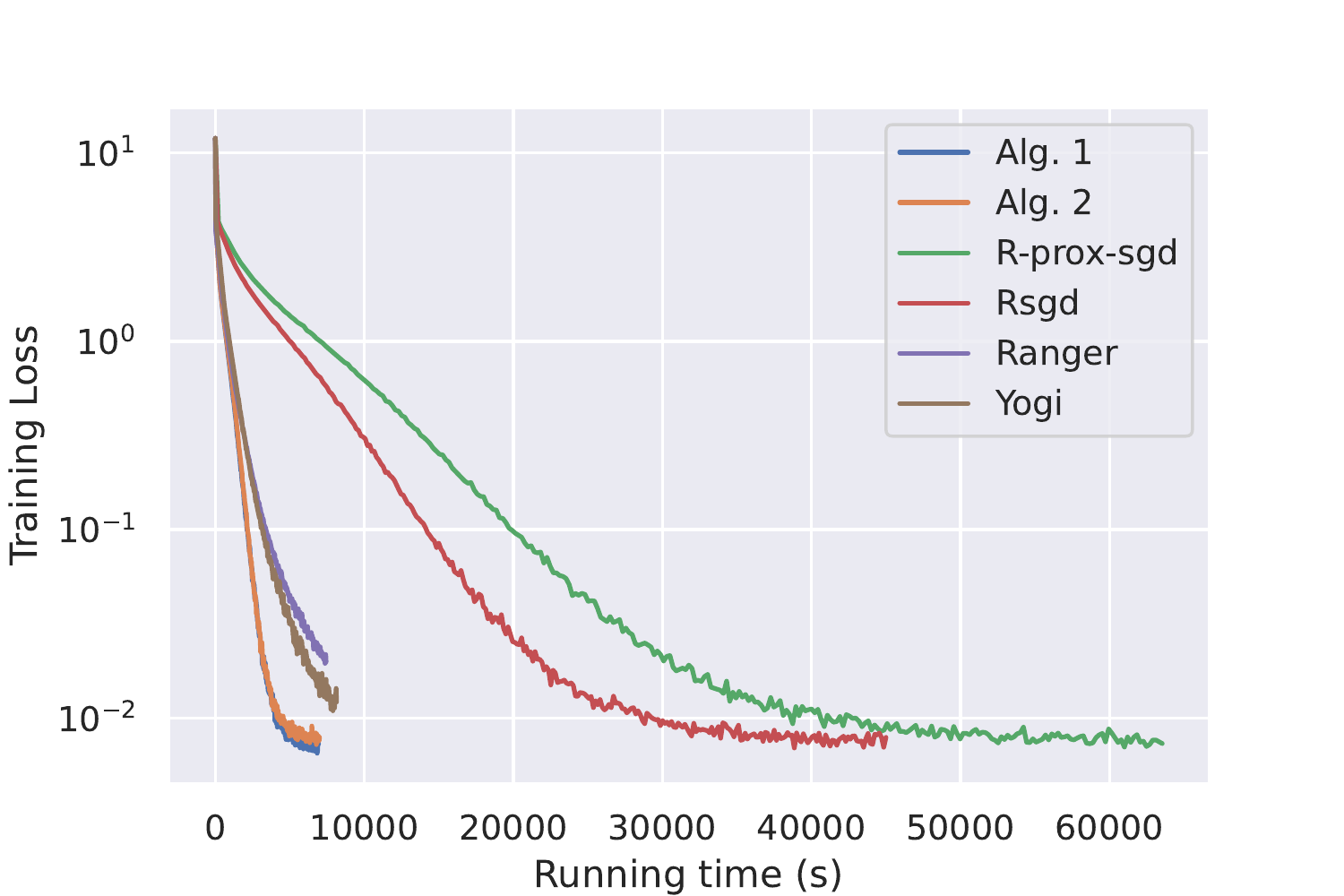}
				\label{Fig:CIFAR100_1_3}
			\end{minipage}%
		}%
		\subfigure[Sparsity]{
			\begin{minipage}[t]{0.33\linewidth}
				\centering
				\includegraphics[width=\linewidth]{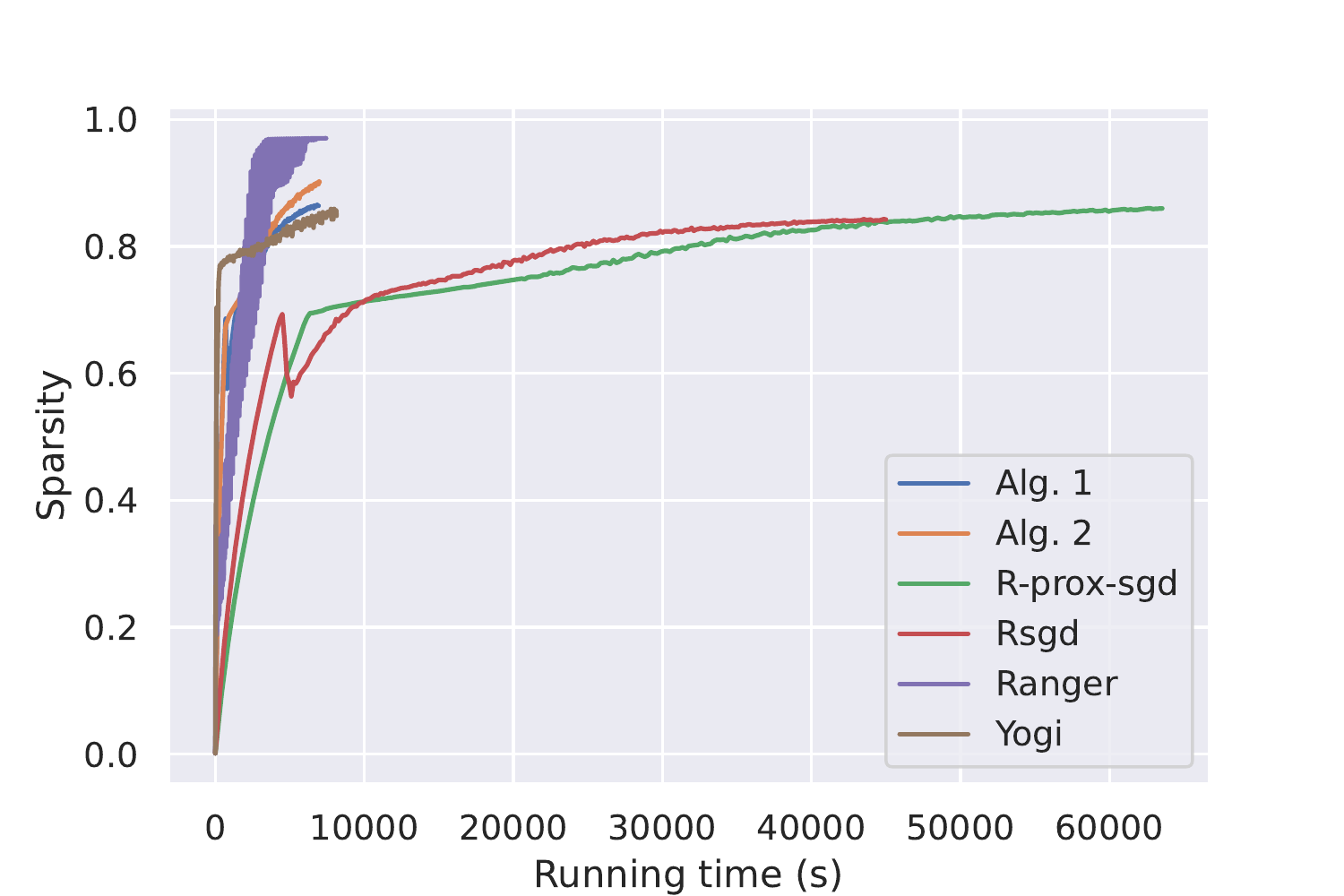}
				\label{Fig:CIFAR100_1_4}
			\end{minipage}%
		}%
		\caption{Comparison of Algorithm \ref{Alg:subgradient_sto} and Algorithm \ref{Alg:subgradient_proximal} with state-of-the-art algorithms R-sgd and R-proxsgd on CIFAR100 dataset.}
		\label{Fig_cifar100}
	\end{figure}

	Figure \ref{Fig_cifar10} and Figure \ref{Fig_cifar100} exhibit the test accuracy, training loss and the sparsity of the parameters generated from the compared algorithms during the training. Here the loss represents the function value of $\phi(W_1,...,W_l, \hat{W})$ in each iteration, and the sparsity refers to the ratio of the entries whose absolute value is smaller than $10^{-5}$ of the trainable parameters of the neural network.  In our comparisons, all the compared algorithms achieve similar test accuracy in the final epoch, while the algorithms employed for minimizing \ref{NEPen} achieves slightly higher sparsity than the Riemannian optimization algorithms. It is worth mentioning that since Ranger is combined with adaptive moment estimation techniques \cite{wright2021ranger21}, it may be more likely to find a sharp minima when compared with SGD method \cite{zhou2020towards}. Therefore, as illustrated in Figure \ref{Fig_cifar10} and Figure \ref{Fig_cifar100}, 
	 applying Ranger to minimize \ref{NEPen} usually achieves much higher sparsity while obtaining similar accuracy when compared with other algorithms.   Moreover, as discussed in Section 4,   R-sgd and R-proxsgd require the computation of retractions, which is more expensive than computing matrix-matrix multiplications. In addition, R-proxsgd requires to solve the Riemannian proximal mapping subproblem, which can be the main bottleneck of the algorithm as discussed in various of existing works \cite{huang2019riemannian,xiao2020l21,xiao2021penalty}. On the other hand, all the algorithms employed for minimizing \ref{NEPen} only require  matrix-matrix multiplications, and Algorithm \ref{Alg:subgradient_proximal} only needs to compute the proximal mapping in $\bb{R}^{n\times p}$, which is usually easy to compute. As illustrated in both Figure \ref{Fig_cifar10} and Figure \ref{Fig_cifar100},  the algorithms employed for minimizing \ref{NEPen} (Algorithm \ref{Alg:subgradient_sto}, Algorithm \ref{Alg:subgradient_proximal}, Ranger and Yogi) achieve significantly superior computational efficiency than R-sgd and R-proxsgd.

	\section{Conclusion}
	
	Nonsmooth nonconvex optimization problems over the Stiefel manifold have real-world applications in various areas. 
	The difficulties in handling such problems often arise from the failure of the chain rule once the required regularity conditions of the objective are missing. 
	We present a novel constraint dissolving function \ref{NEPen} for optimization problems over the Stiefel manifold \ref{Prob_Ori}, whose objective functions are only assumed to be locally Lipschitz continuous. We show that \ref{Prob_Ori} and \ref{NEPen} share the same first-order stationary points and local minimizers in a neighborhood of the Stiefel manifold. The relationships between \ref{Prob_Ori} and \ref{NEPen} illustrate that minimizing a locally Lipschitz smooth function over the Stiefel manifold can be transferred into minimizing \ref{NEPen} without any constraint. Moreover, we show that the Clarke subdifferential of \ref{NEPen} has a explicit formulation from $\partial f$. Therefore, the exactness and accessibility of Clarke subdifferential enable the direct implementation of various unconstrained optimization approaches for solving \ref{Prob_Ori} through \ref{NEPen}.

	We present a representative example to demonstrate that \ref{NEPen} admits direct implementation of existing unconstrained optimization approaches, while their theoretical properties are simultaneously retained. 
	We then introduce a generalized framework for a class of subgradient methods and establish their convergence theories based on \cite{davis2020stochastic,bolte2021conservative}. Based on the proposed framework, we develop a stochastic proximal gradient method for \ref{Prob_Ori}. 
	Furthermore, numerical examples on an orthogonally constrained
	neural network for classification problems shows the superior performance of our proposed algorithms over existing Riemannian optimization approaches. Furthermore, these numerical experiments  demonstrate that the problem can be solved via \ref{NEPen} by efficient and direct implementation of existing unconstrained optimization solvers.

	\bibliography{ref}
	\bibliographystyle{plainnat}

\end{document}